\documentclass[12pt,a4paper,american]{article}
\usepackage{amsmath}
\usepackage{amssymb}

\makeatletter

\newcommand{\lyxaddress}[1]{
\par {\raggedright #1
\vspace{1.4em}
\noindent\par}
}

\usepackage{amsthm}

\usepackage{babel}

\theoremstyle{plain}
\newtheorem{theorem}{Theorem}
  \theoremstyle{definition}
  
  \theoremstyle{remark}
  \newtheorem{remark}[theorem]{Remark}
  \theoremstyle{plain}
  \newtheorem{proposition}[theorem]{Proposition}
  \theoremstyle{plain}
  \newtheorem{lemma}[theorem]{Lemma}
  \theoremstyle{plain}
  \newtheorem{corollary}[theorem]{Corollary}
  \theoremstyle{definition}
  
  \theoremstyle{remark}
  \newtheorem*{remark*}{Remark}
  \theoremstyle{definition}
\newtheorem*{question*}{\it{QUESTION}}
\newcommand{\Dom}{\mathop\mathrm{Dom}\nolimits}

\newcommand{\spec}{\mathop\mathrm{spec}\nolimits}

\renewcommand{\Re}{\mathop\mathrm{Re}\nolimits}
\renewcommand{\Im}{\mathop\mathrm{Im}\nolimits}
\newcommand{\supp}{\mathop\mathrm{supp}\nolimits}

\newcommand{\Res}{\mathop\mathrm{Res}\nolimits}
\makeatother

\usepackage{babel}
\begin{document}

\title{The Nevanlinna parametrization for $q$-Lommel polynomials in the
indeterminate case}

\author{F.~\v{S}tampach$^{1}$, P.~\v{S}\v{t}ov\'\i\v{c}ek$^{2}$}

\date{{}}

\maketitle

\lyxaddress{$^{1}$Department of Applied Mathematics, Faculty of Information
Technology, Czech Technical University in~Prague, Kolejn\'\i~2,
160~00 Praha, Czech Republic}

\lyxaddress{$^{2}$Department of Mathematics, Faculty of Nuclear Science, Czech
Technical University in Prague, Trojanova 13, 12000 Praha, Czech Republic}
\begin{abstract}
\noindent The Hamburger moment problem for the $q$-Lommel polynomials
which are related to the Hahn-Exton $q$-Bessel function is known
to be indeterminate for a certain range of parameters. In this paper,
the Nevanlinna parametrization for the indeterminate case is provided
in an explicit form. This makes it possible to describe all
N-extremal measures of orthogonality. Moreover, a linear and quadratic
recurrence relation are derived for the moment sequence, and the asymptotic
behavior of the moments for large powers is obtained with the aid
of appropriate estimates.
\end{abstract}
\vskip\baselineskip\noindent\emph{ Keywords}: $q$-Lommel polynomials,
Nevanlinna parametrization, measure of orthogonality, moment sequence

\vskip0.5\baselineskip\noindent\emph{ 2010 Mathematical Subject
Classification}: 42C05, 33C47, 33D45

\section{Introduction}

The Lommel polynomials represent a class of orthogonal polynomials
known from the theory of Bessel functions. Several $q$-analogues
of the Lommel polynomials have been introduced and studied in \cite{KoelinkSwarttouw,KoelinkVanAssche,Koelink}.
One of the three commonly used $q$-analogues of the Bessel function
of the first kind is known as the Hahn-Exton $q$-Bessel function
(sometimes also called the third Jackson $q$-Bessel function or $_{1}\phi_{1}$
$q$-Bessel function). It is defined by
\begin{equation}
J_{\nu}(z;q)=z^{\nu}\,\frac{(q^{\nu+1};q)_{\infty}}{(q;q)_{\infty}}\,\,_{1}\phi_{1}(0;q^{\nu+1};q,qz^{2}).\label{eq:def_HE_q_Bessel}
\end{equation}
It is of importance that $J_{\nu}(z;q)$ satisfies the recurrence
relation
\[
J_{\nu+1}(z;q)-\left(z+\frac{1-q^{\nu}}{z}\right)J_{\nu}(z;q)+J_{\nu-1}(z;q)=0.
\]
By iterating this rule one arrives at the formula 
\begin{equation}
J_{\nu+n}(z;q)=h_{n,\nu}(z^{-1};q)J_{\nu}(z;q)-h_{n-1,\nu+1}(z^{-1};q)J_{\nu-1}(z;q)\label{eq:qLommel_qBessel}
\end{equation}
where $h_{m,\nu}(w;q)$ are polynomials in $q^{\nu}$ and Laurent
polynomials in $w$, see \cite{KoelinkVanAssche} for more details.
This is a familiar situation, with equation (\ref{eq:qLommel_qBessel})
being analogous to the well known relation between the Lommel polynomials
and the Bessel functions, cf. \cite[Chapter~9]{Watson}. Thus the
polynomials $h_{m,\nu}(w;q)$ can be referred to as the $q$-Lommel
polynomials.

On one hand, the polynomials $h_{n,\nu}(w;q)$ can be treated as orthogonal
Laurent polynomials in the variable $w$.  The corresponding orthogonality
relation has been described in \cite{KoelinkVanAssche}. On the other hand,
$h_{n,\nu}(w;q)$ are also orthogonal polynomials in the variable $q^{\nu}$. In
Theorem~3.6 and Corollary~3.7 in \cite{Koelink}, Koelink described a
corresponding measure of orthogonality. It turns out that the measure of
orthogonality is supported on the zeros of the Hahn-Exton $q$-Bessel function
considered as a function of the order $\nu$ . Moreover, the measure of
orthogonality is unique if $w^{-2}\leq q$ or $w^{-2}\geq q^{-1}$. For
$q<w^{-2}<q^{-1}$, however, the corresponding Hamburger moment problem is
indeterminate and so there exist infinitely many measures of
orthogonality. The measure described in \cite{Koelink} represents a Nevanlinna
(or N-) extremal solution of the indeterminate Hamburger moment problem, and
it can be seen to correspond to the Friedrichs extension of the underlying
Jacobi matrix operator.

Let us also remark that the $q$-Lommel polynomials admit another
interpretation in the framework of a birth and death process with
exponentially growing birth and death rates. More precisely, the birth
rate is supposed to be $\lambda_{n}=w^{-2}q^{-n}$ while the death
rate is $\mu_{n}=q^{-n}$ (or vice versa). See, for example, \cite{Ismail_etal}
for more information on the subject.

As already pointed out in \cite{Koelink}, it is of interest and in
fact a fundamental question to determine all possible measures of
orthogonality in terms of the Nevanlinna parametrization. An explicit
solution of this problem becomes the main goal of the current paper.
To achieve it we heavily rely on the knowledge of the generating function
for the $q$-Lommel polynomials. Having the Nevanlinna parametrization
at hand it is straightforward to describe all N-extremal measures
of orthogonality. The case when $w=1$ turns out to be somewhat special
and requires additional efforts though no new ideas are in principle
needed. To our best knowledge, formulas for this particular case have
been omitted in the past research works on the $q$-Lommel polynomials.

The measures of orthogonality we are going to describe are necessarily
discrete. To reveal a bit their structure we have a closer look at the
asymptotics of the mass points and the corresponding weights of such a
measure. To this end, we make use of some known results concerned with the
asymptotic behavior of the roots of the $q$-Bessel functions. A brief summary
of basic facts and references on this subject is provided in
Appendix. Moreover, we were able to complete these facts with some additional
details.

Furthermore, we pay some attention to the sequence of moments related to the
$q$-Lommel polynomials. By Favard's theorem, the moments are uniquely
determined by the coefficients in the recurrence relation for the $q$-Lommel
polynomials and otherwise they are independent of a particular choice of the
measure of orthogonality in the indeterminate case. It does not seem that the
moment sequence can be found explicitly.  We provide at least a linear and
quadratic recurrence relation for it and describe qualitatively its asymptotic
behavior for large powers.

Let us note that throughout the whole paper the parameter $q$ is
assumed to satisfy $0<q<1$. Furthermore, as far as the basic (or
$q$-) hypergeometric series are concerned, as well as other $q$-symbols
and functions, we follow the notation of Gasper and Rahman \cite{GasperRahman}.

\section{The Nevanlinna functions for $q$-Lommel polynomials}

\subsection{The $q$-Lommel polynomials}

In the current paper we prefer to work directly with the $\,_{1}\phi_{1}$
basic hypergeometric function and do not insist on its interpretation
as the $q$-Bessel function in accordance with (\ref{eq:def_HE_q_Bessel}).
This leads us to using a somewhat modified notation if compared to
that usually used in connection with $q$-Bessel functions, for instance,
in \cite{Koelink}. Moreover, the notation used in this paper may
stress some similarity of the Hamburger moment problem for the $q$-Lommel
polynomials with the same problem for the Al-Salam-Carlitz II polynomials.
The Hamburger moment problem is actually known to be indeterminate
for particular values of parameters in both cases but there are also
some substantial differences, see \cite[Section~4]{BergValent}.

Thus we write $a>0$ instead of $w^{-2}$ and $x\in\mathbb{C}$ instead
of $q^{\nu}$. The basic recurrence relation we are going to study,
defining a sequence of monic orthogonal polynomials $\{F_{n}(a,q;x)\}_{n=0}^{\infty}$
(in the variable $x$ and depending on two parameters $a$ and $q$),
reads
\begin{equation}
u_{n+1}=\big(x-(a+1)q^{-n}\big)u_{n}-aq^{-2n+1}u_{n-1},\ \ n\in\mathbb{Z}_{+}\label{eq:recur_def_F_n}
\end{equation}
($\mathbb{Z}_{+}$ standing for nonnegative integers). As usual, the
initial conditions are imposed in the form $F_{-1}(a,q;x)=0$ and
$F_{0}(a,q;x)=1$. In order to be able to compare some results derived
below with the already known results on the $q$-Lommel polynomials
let us remark that the $q$-Lommel polynomials $h_{n,\nu}(w;q)$ introduced
in (\ref{eq:qLommel_qBessel}) are related to the monic polynomials
$F_{n}(a,q;x)$ by the formula
\[
h_{n,\nu}(w;q)=(-1)^{n}w^{n}q^{n(n-1)/2}F_{n}(w^{-2},q;q^{\nu}).
\]

From (\ref{eq:recur_def_F_n}) one immediately deduces the symmetry
property
\[
a^{n}F_{n}(a^{-1},q;x)=F_{n}(a,q;ax),\quad n\in\mathbb{Z}_{+}.
\]
This suggests that one can restrict values of the parameter $a$ to
the interval $0<a<1$. We usually try, however, to formulate our results
for both cases, $a<1$ and $a>1$, for the sake of completeness. The
case $a=1$ is somewhat special and should be treated separately.

Letting 
\begin{equation}
G_{n}(a,q;x)=q^{1-n}F_{n-1}(a,q;qx),\ \ n\in\mathbb{Z}_{+},\label{eq:G_rel_F}
\end{equation}
we get a second linearly independent solution of (\ref{eq:recur_def_F_n}),
a sequence of monic polynomials $\{G_{n}(a,q;x)\}$ fulfilling the
initial conditions $G_{0}(a,q;x)=0$ and $G_{1}(a,q;x)=1$. Normalizing
the monic polynomials $F_{n}(a,q;x)$ we get an orthonormal polynomial
sequence $\{P_{n}(a,q;x)\}_{n=0}^{\infty}$. Explicitly, 
\begin{equation}
  P_{n}(a,q;x)=a^{-n/2}q^{n^{2}/2}F_{n}(a,q;x),\ \ n\in\mathbb{Z}_{+}.
  \label{eq:P_n_rel_F_n}
\end{equation}
The polynomials of the second kind, $Q_{n}(a,q;x)$, are related to
the monic polynomials $G_{n}(a,q;x)$ by a similar equality,
\begin{equation}
Q_{n}(a,q;x)=a^{-n/2}q^{n^{2}/2}G_{n}(a,q;x),\ \ n\in\mathbb{Z}_{+},\label{eq:Qn_rel_Gn}
\end{equation}
and obey the initial conditions $Q_{0}(a,q;x)=0$, $Q_{1}(a,q;x)=\sqrt{q/a}$.

Note that polynomials $P_{n}(a,q;x)$ solve the second-order difference
equation 
\[
\sqrt{a}q^{-n+1/2}v_{n-1}+\big((a+1)q^{-n}-x\big)v_{n}+\sqrt{a}q^{-n-1/2}v_{n+1}=0,\ \ n\in\mathbb{Z}_{+},
\]
with the initial conditions $P_{-1}(a,q;x)=0$ and $P_{0}(a,q;x)=1$.
Denote by $\alpha_{n}$ and $\beta_{n}$ the coefficients in this
difference equation,
\begin{equation}
  \alpha_{n}=a^{1/2}q^{-n-1/2},\ \beta_{n}=(a+1)q^{-n},\ \ n\in\mathbb{Z}_{+}.
  \label{eq:alpha_beta}
\end{equation}
The difference equation can be interpreted as the formal eigenvalue
equation for the Jacobi matrix 
\begin{equation}
J=J(a,q)=\begin{pmatrix}\beta_{0} & \alpha_{0}\\
\alpha_{0} & \beta_{1} & \alpha_{1}\\
 & \alpha_{1} & \beta_{2} & \alpha_{2}\\
 &  & \ddots & \ddots & \ddots
\end{pmatrix}\!.
\label{eq:J}
\end{equation}
Then $(P_{0}(x),P_{1}(x),P_{2}(x),\ldots)$ is a formal eigenvector
(where $P_{j}(x)\equiv P_{j}(a,q;x)$). Let us emphasize that $J$
is positive on the subspace in $\mathbb{\ell}^{2}(\mathbb{Z}_{+})$
formed by sequences with only finitely many nonzero entries, i.e.
on the linear hull of the canonical basis in $\mathbb{\ell}^{2}(\mathbb{Z}_{+})$.
Actually, it is not difficult to verify that for every $N\in\mathbb{Z}_{+}$
and $\xi\in\mathbb{R}^{N+1}$,
\begin{equation}
\sum_{n=0}^{N}\beta_{n}\xi_{n}^{\,2}+2\sum_{n=0}^{N-1}\alpha_{n}\xi_{n}\xi_{n+1}=a\xi_{0}^{\,2}+q^{-N}\xi_{N}^{\,2}+\sum_{n=0}^{N-1}q^{-n}\left(\!\left(\frac{a}{q}\right)^{\!1/2}\xi_{n+1}+\xi_{n}\right)^{\!2}\geq0.\label{eq:quadraticform_J}
\end{equation}

Recurrence (\ref{eq:recur_def_F_n}) can be solved explicitly in the
particular case when $x=0$. One finds that
\[
F_{n}(a,q;0)=(-1)^{n}q^{-n(n-1)/2}\,\frac{1-a^{n+1}}{1-a},\ G_{n}(a,q;0)=(-1)^{n+1}q^{-n(n-1)/2}\,\frac{1-a^{n}}{1-a},
\]
for $n\in\mathbb{Z}_{+}$ and $a\neq1$. Consequently,
\begin{equation}
P_{n}(a,q;0)=(-1)^{n}q^{n/2}a^{-n/2}\,\frac{1-a^{n+1}}{1-a},\ Q_{n}(a,q;0)=(-1)^{n+1}q^{n/2}a^{-n/2}\,\frac{1-a^{n}}{1-a}.\label{eq:P_n_and_Q_n_x=00003D0}
\end{equation}
The quantities $P_{n}(1,q;0)$ and $Q_{n}(1,q;0)$ can be obtained
from (\ref{eq:P_n_and_Q_n_x=00003D0}) in the limit $a\rightarrow1$,
\[
P_{n}(1,q;0)=(-1)^{n}q^{n/2}(n+1),\ Q_{n}(1,q;0)=(-1)^{n+1}q^{n/2}n.
\]

\subsection{The generating function}

A formula for the generating function for the $q$-Lommel polynomials
has been derived in \cite[Eq. (4.22)]{KoelinkSwarttouw}. Here we
reproduce the formula and provide its proof since it is quite crucial
for the computations to follow of the Nevanlinna functions $A$, $B$,
$C$, and $D$.

\begin{proposition}
\label{thm:generate_fce}
Let $a>0$. The generating function for the polynomials $F_{n}(a,q;x)$ equals
\begin{equation}
  \sum_{n=0}^{\infty}q^{n(n-1)/2}F_{n}(a,q;x)(-t)^{n}
  =\sum_{k=0}^{\infty}\frac{q^{k(k-1)/2}(-xt)^{k}}{(t;q)_{k+1}(at;q)_{k+1}}
  =\frac{\,_{2}\phi_{2}(q,0;qt,qat;q,xt)}{(1-t)(1-at)}
  \label{eq:gener_func_F_n}
\end{equation}
where $|t|<\min(1,a^{-1})$.
\end{proposition}

\begin{proof} The last equality in (\ref{eq:gener_func_F_n}) is
obvious from the definition of the basic hypergeometric function.
Suppose $a$ and $x$ being fixed and put
\[
V(t)=\sum_{k=0}^{\infty}\frac{q^{k(k-1)/2}(-xt)^{k}}{(t;q)_{k+1}(at;q)_{k+1}}\,.
\]
$V(t)$ is a well defined analytic function for $|t|<\min(1,a^{-1})$
which is readily seen to satisfy the $q$-difference equation
\begin{equation}
(1-t)(1-at)V(t)=1-xtV(qt).\label{eq:q-diff_V}
\end{equation}
Writing the power series expansion of $V(t)$ at $t=0$ in the form
\[
V(t)=\sum_{n=0}^{\infty}u_{n}q^{n(n-1)/2}(-t)^{n}
\]
and inserting the series into (\ref{eq:q-diff_V}) one finds that
the coefficients $u_{n}$ obey the recurrence (\ref{eq:recur_def_F_n})
and the initial conditions $u_{0}=1$, $u_{1}=-1-a+x$. Necessarily,
$u_{n}=F_{n}(a,q;x)$ for all $n\in\mathbb{Z}_{+}$. \end{proof}

In \cite[Section~4]{KoelinkSwarttouw} and particularly in \cite[Eq. (2.6)]{Koelink}
there is stated an explicit formula for the polynomials $F_{n}(a,q;x)$,
namely
\[
F_{n}(a,q;x)=(-1)^{n}q^{-n(n-1)/2}\sum_{j=0}^{n}\frac{q^{jn}(q^{-n};q)_{j}}{(q;q)_{j}}\,\,_{2}\phi_{1}(q^{j-n},q^{j+1};q^{-n};q,q^{-j}a)\, x^{j}.
\]
Let us restate this formula as an immediate corollary of Proposition~\ref{thm:generate_fce}.

\begin{corollary} The polynomials $F_{n}(a,q;x)$, $n\in\mathbb{Z}_{+}$,
can be expressed explicitly as follows
\begin{equation}
F_{n}(a,q;x)=(-1)^{n}q^{-n(n-1)/2}\sum_{j=0}^{n}\frac{(-1)^{j}q^{j(j-1)/2}}{(q;q)_{j}^{\,2}}\!\left(\,\sum_{k=0}^{n-j}(q^{k+1};q)_{j}(q^{n-j-k+1};q)_{j}\, a^{k}\!\right)\! x^{j}.\label{eq:F_n_explicit}
\end{equation}
\end{corollary}

\begin{proof} The formula can be derived by equating the coefficients
of equal powers of $t$ in (\ref{eq:gener_func_F_n}). To this end,
one has to apply the $q$-binomial formula 
\[
\frac{1}{(z;q)_{k}}={}_{1}\phi_{0}(q^{k};\,;q;z)=\sum_{n=0}^{\infty}\frac{(q^{k};q)_{n}}{(q;q)_{n}}\, z^{n},\ \ |z|<1,
\]
cf. \cite[Eq. (II.3)]{GasperRahman}. \end{proof}

\subsection{\label{sec:general_nevan} The indeterminate case and the Nevanlinna
parametrization}

We are still assuming that $a$ is positive. In \cite[Lemma~3.1]{Koelink} it is
proved that the Hamburger moment problem for the orthogonal polynomials
$F_{n}(a,q;x)$ (or $P_{n}(a,q;x)$) is indeterminate if and only if
$q<a<q^{-1}$. This is, however, clear from formulas
(\ref{eq:P_n_and_Q_n_x=00003D0}) and from the well known criterion
(cf. Addenda and Problems 10. to Chapter~2 in \cite{Akhiezer}) according to
which the Hamburger moment problem is indeterminate if and only if
\[
\sum_{n=0}^{\infty}\left(P_{n}(a,q;0)^{2}+Q_{n}(a,q;0)^{2}\right)<\infty.
\]
This also means that the Jacobi matrix operator $J$ defined in (\ref{eq:J}),
(\ref{eq:alpha_beta}), with $\Dom J$ equal to the linear hull of
the canonical basis in $\ell^{2}(\mathbb{Z}_{+})$, is not essentially
self-adjoint if and only if $a\in(q,q^{-1})$. In this case, the deficiency indices 
of $J$ are $(1,1)$, see~\cite[Chapter~4]{Akhiezer}.

Hence for $q<a<q^{-1}$ there exist infinitely many distinct measures
of orthogonality parametrized with the aid of the Nevanlinna functions
$A$, $B$, $C$, and $D$,
\begin{eqnarray*}
A(z)=z\sum_{n=0}^{\infty}Q_{n}(0)Q_{n}(z),\hskip18pt &  & B(z)=-1+z\sum_{n=0}^{\infty}Q_{n}(0)P_{n}(z),\\
C(z)=1+z\sum_{n=0}^{\infty}P_{n}(0)Q_{n}(z), &  & D(z)=z\sum_{n=0}^{\infty}P_{n}(0)P_{n}(z),
\end{eqnarray*}
where $P_{n}$ and $Q_{n}$ are the polynomials of the first and second
kind, respectively \cite{Akhiezer,ShohatTamarkin}. All these Nevanlinna
functions are entire and
\begin{equation}
A(z)D(z)-B(z)C(z)=1,\ \ \forall z\in\mathbb{C}.\label{eq:AD-BC=00003D1}
\end{equation}
According to the Nevanlinna theorem, all measures of orthogonality
$\mu_{\varphi}$ for which the set $\{P_{n};\, n\in\mathbb{Z}_{+}\}$
is orthonormal in $L^{2}(\mathbb{R},\mbox{d}\mu_{\varphi})$, are
in one-to-one correspondence with functions $\varphi$ belonging to
the one-point compactification $\mathcal{P}\cup\{\infty\}$ of the
space of Pick functions $\mathcal{P}$. Recall that Pick functions
are defined and holomorphic on the open complex halfplane $\Im z>0$,
with values in the closed halfplane $\Im z\geq0$. The correspondence
is established by identifying the Stieltjes transform of the measure
$\mu_{\varphi}$,
\begin{equation}
\int_{\mathbb{R}}\frac{\mbox{d}\mu_{\varphi}(x)}{z-x}=\frac{A(z)\varphi(z)-C(z)}{B(z)\varphi(z)-D(z)}\,,\ \ z\in\mathbb{C}\setminus\mathbb{R}.\label{eq:Nevan_param_Stieltj_transf}
\end{equation}

By a theorem due to M.~Riesz, $\{P_{n};\, n\in\mathbb{Z}_{+}\}$ is an
orthonormal basis in $L^{2}(\mathbb{R},\mbox{d}\mu_{\varphi})$ if and only if
$\varphi=t$ is a constant function with $t\in\mathbb{R}\cup\{\infty\}$
\cite[Theorem~2.3.3]{Akhiezer}. Then the measure $\mu_{t}$ is said to be
N-extremal. Moreover, the N-extremal measures $\mu_{t}$ are in one-to-one
correspondence with the self-adjoint extensions $T_{t}$ of the Jacobi operator
$J$ mentioned above. In more detail, if $E_{t}$ is the spectral measure of
$T_{t}$ and $e_{0}$ is the first vector of the canonical basis in
$\ell^{2}(\mathbb{Z}_{+})$ then
$\mu_{t}=\langle e_{0},E_{t}(\cdot)e_{0}\rangle$
\cite[Chapter~4]{Akhiezer}. The operators $T_{t}$ in the indeterminate case
are known to have a compact resolvent. Hence any N-extremal measure $\mu_{t}$
is purely discrete and supported on $\spec T_{t}$.

On the other hand, referring to (\ref{eq:Nevan_param_Stieltj_transf}),
the support of $\mu_{t}$ is also known to be equal to the zero set
\begin{equation}
\mathfrak{Z}_{t}=\{x\in\mathbb{R};\, B(x)t-D(x)=0\}\label{eq:Zero_t}
\end{equation}
\cite[Section~2.4]{Akhiezer}. Hence
\begin{equation}
\mu_{t}=\sum_{x\in\mathfrak{Z}_{t}}\rho(x)\delta_{x}\label{eq:mu_t}
\end{equation}
where $\rho(x)=\mu_{t}(\{x\})$ and $\delta_{x}$ is the Dirac measure
supported on $\{x\}$. Equation (\ref{eq:Nevan_param_Stieltj_transf}),
with $\varphi=t$, is nothing but the Mittag-Leffler expansion of
the meromorphic function on the right-hand side,
\[
\sum_{x\in\mathfrak{Z}_{t}}\frac{\rho(x)}{z-x}=\frac{A(z)t-C(z)}{B(z)t-D(z)}\,,
\]
 cf. \cite[footnote on p.~55]{Akhiezer}. From here it can be deduced
that
\begin{equation}
\rho(x)=\Res_{z=x}\frac{A(z)t-C(z)}{B(z)t-D(z)}=\frac{A(x)t-C(x)}{B'(x)t-D'(x)}=\frac{1}{B'(x)D(x)-B(x)D'(x)}\label{eq:rho}
\end{equation}
since, for $x\in\mathfrak{Z}_{t}$, $t=D(x)/B(x)$.

It should be noted that we are dealing with the Stieltjes case for
the matrix operator $J$ is positive on its domain of definition,
see (\ref{eq:quadraticform_J}). This means that, for any choice of
parameters from the specified range, there always exists a measure
of orthogonality with its support contained in $[0,+\infty)$. In
particular, if $a\in(q,q^{-1})$ then at least one of the measures
of orthogonality is supported by $[0,+\infty)$. From \cite[Lemma~1]{Chihara68}
it is seen that the limit
\begin{equation}\label{eq:def_alpha}
\lim_{n\to\infty}\,\frac{P_{n}(0)}{Q_{n}(0)}=\alpha\in(-\infty,0\,]
\end{equation}
exists. And, as explained in \cite[Remark~2.2.2]{BergValent}, an N-extremal
measure of orthogonality $\mu_{t}$ is supported by $[\,0,\infty)$ if and only
if $t\in[\,\alpha,0\,]$, the Stieltjes moment problem is determinate for
$\alpha=0$ and indeterminate for $\alpha<0$. Let us note that $\mu_{0}$ is the
unique N-extremal measure for which $0$ is a mass point.

In our case, making once more use of the explicit form (\ref{eq:P_n_and_Q_n_x=00003D0}),
we have 
\begin{equation}
\alpha=\lim_{n\rightarrow\infty}\,\frac{P_{n}(a,q;0)}{Q_{n}(a,q;0)}=\begin{cases}
-1, & \mbox{ if }a\in(0,1],\\
-a, & \mbox{ if }a>1.
\end{cases}\label{eq:quant_alpha}
\end{equation}
Hence the Stieltjes problem is indeterminate for any value $a\in(q,q^{-1})$.

The self-adjoint operator $T_{\alpha}$ corresponding to the N-extremal
measure $\mu_{\alpha}$ is nothing but the Friedrichs extension of
$J$ \cite[Proposition~3.2]{Pedersen}. The parameter $\alpha$ can
also be computed in the limit
\[
\alpha=\lim_{x\to-\infty}\frac{D(x)}{B(x)}\,,
\]
and by inspection of the function $D(x)/B(x)$ one finds that $\mu_{t}$
has exactly one negative mass point if $t\notin[\,\alpha,0\,]$ including
$t=\infty$ \cite[Lemma~2.2.1]{BergValent}. It is known, too, that
Markov's theorem applies in the indeterminate Stieltjes case meaning
that
\begin{equation}
\lim_{n\to\infty}\,\frac{Q_{n}(z)}{P_{n}(z)}=\frac{A(z)\alpha-C(z)}{B(z)\alpha-D(z)}\,,\ z\in\text{\ensuremath{\mathbb{C}}}\setminus\supp(\mu_{\alpha})\label{eq:lim_Qnz_over_Pnz}
\end{equation}
\cite[Theorem~2.1]{Berg}. In addition, in the same case, one has
the limit
\begin{equation}
\lim_{n\to\infty}\,\frac{P_{n}(z)}{Q_{n}(0)}=D(z)-B(z)\alpha,\ z\in\text{\ensuremath{\mathbb{C}}},\label{eq:eq:lim_Pnz_over_Qn0}
\end{equation}
as derived in \cite{Chihara68} and also in \cite{Pedersen}.

Further we wish to recall yet another interesting application of the
Nevanlinna functions. It is shown in \cite{BuchwalterCassier} that
the reproducing kernel can be expressed in terms of functions $B(z)$
and $D(z)$,
\begin{equation}
K(u,v):=\sum_{n=0}^{\infty}P_{n}(u)P_{n}(v)=\frac{B(u)D(v)-D(u)B(v)}{u-v}\,,\label{eq:repro_kernel}
\end{equation}
see also \cite[Section~1]{Berg95}.

Finally, let us note that Krein considered a slightly different
parametrization of the set of solutions to an indeterminate Stieltjes moment
problem, see \cite[Chapter~V, \S 5]{KreinNudelman}. The Krein parametrization
uses four entire functions $a$, $b$, $c$, and $d$ which can be expressed in
terms of the Nevanlina functions,
\begin{eqnarray}
  && a(z)=A(-z)-\frac{C(-z)}{\alpha},\qquad
  b(z)=-B(-z)+\frac{D(-z)}{\alpha}, \nonumber\\
  && c(z)=C(-z), \hskip 75pt d(z)=-D(-z), 
  \label{eq:Krein_abcd} 
\end{eqnarray}
with $\alpha$ being defined in (\ref{eq:def_alpha}); see also
\cite[Section~5]{Berg95}. The solutions $\mu_\sigma$ are in one-to-one
correspondence with functions $\sigma$ from the one-point compactification of
the Stieltjes class $\mathcal{S}^{-}$, see \cite{KreinNudelman} for details.
The correspondence is established by the formula
\[
 \int_{0}^{\infty}\frac{\mbox{d}\mu_{\sigma}(x)}{z+x}=\frac{a(z)+c(z)\sigma(z)}{b(z)+d(z)\sigma(z)}\,,\ \ z\in\mathbb{C}\setminus(-\infty,0],
\]
where $\sigma\in\mathcal{S}^{-}\cup\{\infty\}$.

\subsection{An explicit form of the Nevanlinna functions}

In order to describe conveniently the Nevanlinna parametrization in
the studied case we introduce a shorthand notation for particular
basic hypergeometric series while not indicating the dependance on
$q$ explicitly. We put
\begin{equation}
\varphi_{a}(z)={}_{1}\phi_{1}(0;qa;q,z),\ \psi_{a}(z)={}_{1}\phi_{1}(0;qa^{-1};q,a^{-1}z),\label{eq:def_phi_a_psi_a}
\end{equation}
and
\[
\chi_{1}(z)=\frac{\partial}{\partial p}\,{}_{1}\phi_{1}(0;p;q,z)\Big|_{p=q}.
\]

\begin{theorem}
\label{thm:Nevan_ABCD}
Let $1\neq a\in(q,q^{-1})$.  Then the entire functions $A$, $B$, $C$ and $D$
from the Nevanlinna parametrization are as follows:
\begin{eqnarray}
  &  & A(a,q;z)=\frac{\varphi_{a}(qz)
    -\psi_{a}(qz)}{1-a}\,,\ B(a,q;z)=\frac{a\psi_{a}(z)-\varphi_{a}(z)}{1-a}\,,
  \nonumber \\
  &  & C(a,q;z)=\frac{\psi_{a}(qz)-a\varphi_{a}(qz)}{1-a}\,,\ D(a,q;z)
  =\frac{a\big(\varphi_{a}(z)-\psi_{a}(z)\big)}{1-a}\,.
  \label{eq:Nevan_ABCD}
\end{eqnarray}
For $a=1$ these functions take the form
\begin{eqnarray}
  &  & A(1,q;z)=-2q\,\chi_{1}(qz)-z\,\frac{\partial}{\partial z}\,
  \varphi_{1}(qz),\ B(1,q;z)
  =2q\,\chi_{1}(z)+z^{2}\frac{\partial}{\partial z}
  \big(z^{-1}\varphi_{1}(z)\big),\quad\nonumber \\
  &  & C(1,q;z)=2q\,\chi_{1}(qz)
  +\frac{\partial}{\partial z}\big(z\varphi_{1}(qz)\big),\ D(1;q,z)
  =-2q\,\chi_{1}(z)-z\frac{\partial}{\partial z}\,\varphi_{1}(z).
  \label{eq:Nevan_ABCD_a1}
\end{eqnarray}
\end{theorem}

\begin{proof} We shall confine ourselves to computing the function
$A$ only. The formulas for $B$, $C$ and $D$ can be derived in
a fully analogous manner. Starting from the definition of $A$ and
recalling formulas (\ref{eq:P_n_and_Q_n_x=00003D0}) and (\ref{eq:Qn_rel_Gn}),
(\ref{eq:G_rel_F}) for $Q_{n}(a,q;0)$ and $Q_{n}(a,q;x)$, respectively,
one has
\begin{eqnarray*}
 &  & A(a,q;z)=\frac{qz}{1-a}\sum_{n=1}^{\infty}(-1)^{n+1}(a^{-n}-1)q^{n(n-1)/2}F_{n-1}(a,q;qz)\\
 &  & =\,\frac{zq}{1-a}\!\left(a^{-1}\sum_{n=0}^{\infty}q^{n(n-1)/2}F_{n}(a,q;qz)(-qa^{-1})^{n}-\sum_{n=0}^{\infty}q^{n(n-1)/2}F_{n}(a,q;qz)(-q)^{n}\right)\!.
\end{eqnarray*}
From comparison of both sums in the last expression with formula (\ref{eq:gener_func_F_n})
for the generating function it becomes clear that the sums can be
expressed in terms of basic hypergeometric functions, namely
\begin{eqnarray*}
A(a,q;z) & = & \frac{qz}{1-a}\!\left(\frac{a^{-1}\,_{2}\phi_{2}(0,q;q^{2}a^{-1},q^{2};q,q^{2}a^{-1}z)}{(1-qa^{-1})(1-q)}-\frac{\,_{2}\phi_{2}(0,q;q^{2},aq^{2};q,zq^{2})}{(1-q)(1-qa)}\right)\\
 & = & \frac{1}{1-a}\!\left(\big(1-{}_{1}\phi_{1}(0;qa^{-1};q,qa^{-1}z)\big)-\big(1-{}_{1}\phi_{1}(0;qa;q,qz)\big)\right)\!.
\end{eqnarray*}
Thus one arrives at the first equation in (\ref{eq:Nevan_ABCD}).

Concerning the particular case $a=1$, formulas (\ref{eq:Nevan_ABCD_a1})
can be derived by applying the limit $a\rightarrow1$ to formulas
(\ref{eq:Nevan_ABCD}). This is actually possible since Proposition~2.4.1
and Remark~2.4.2 from \cite{BergValent} guarantee that the functions
$A(a,q;z)$, $B(a,q;z)$, $C(a,q;z)$, $D(a,q;z)$ depend continuously
on $a\in(q,q^{-1})$. In order to be able to apply this theoretical
result one has to note that the coefficients in the recurrence (\ref{eq:recur_def_F_n})
depend continuously on $a$, and to verify that the series $\sum_{n=0}^{\infty}P_{n}(a,q;0)^{2}$
and $\sum_{n=0}^{\infty}Q_{n}(a,q;0)^{2}$ converge uniformly for
$a$ in compact subsets of $(q,q^{-1})$. But the latter fact is obvious
from (\ref{eq:P_n_and_Q_n_x=00003D0}).

For instance, in case of function $A$ one finds that 
\[
A(1,q;z)=\lim_{a\rightarrow1}\frac{\varphi_{a}(qz)-\psi_{a}(qz)}{1-a}=-\frac{\partial}{\partial a}\big(\varphi_{a}(qz)-\psi_{a}(qz)\big)\Big|_{a=1}.
\]
A straightforward computation yields the first equation in (\ref{eq:Nevan_ABCD_a1}),
and similarly for the remaining three equations. \end{proof}

\begin{corollary} The following limits are true:
\begin{eqnarray}
\hskip-1.5em &  & \lim_{n\rightarrow\infty}(-1)^{n}\left(\frac{a}{q}\right)^{\! n/2}P_{n}(a,q;x)=\frac{\,_{1}\phi_{1}(0;qa;q,x)}{1-a}\,,\hskip3em\ \mbox{ if}\ \ q<a<1,\nonumber \\
\hskip-1.5em &  & \lim_{n\rightarrow\infty}(-1)^{n}(qa)^{-n/2}P_{n}(a,q;x)=\frac{a\,\,_{1}\phi_{1}(0;qa^{-1};q,a^{-1}x)}{a-1}\,,\ \mbox{ if}\ \ 1<a<q^{-1},\label{eq:lim_rel_F_n_all_a}\\
\hskip-1.5em &  & \lim_{n\rightarrow\infty}\frac{(-1)^{n}}{n}\, q^{-n/2}P_{n}(1,q;x)=\,_{1}\phi_{1}(0;q;q,x),\nonumber 
\end{eqnarray}
and
\begin{equation}
\lim_{n\to\infty}\,\frac{Q_{n}(z)}{P_{n}(z)}=\begin{cases}
{\displaystyle -\frac{_{1}\phi_{1}(0;qa;q,qz)}{_{1}\phi_{1}(0;qa;q,z)}} & \text{for }q<a<1,\\
\noalign{\bigskip}{\displaystyle -\frac{_{1}\phi_{1}(0;qa^{-1};q,qa^{-1}z)}{a\,{}_{1}\phi_{1}(0;qa^{-1};q,a^{-1}z)}} & \text{for }1<a<q^{-1},\\
\noalign{\bigskip}{\displaystyle -\frac{_{1}\phi_{1}(0;q;q,qz)}{_{1}\phi_{1}(0;q;q,z)}} & \text{for }a=1.
\end{cases}\label{eq:cases_lim_Qnz_Pnz}
\end{equation}
\end{corollary}

\begin{proof} From (\ref{eq:Nevan_ABCD}), (\ref{eq:Nevan_ABCD_a1})
and (\ref{eq:quant_alpha}) one immediately infers that
\begin{eqnarray*}
C(a,q;z)-A(a,q;z)\alpha & = & \varphi_{a}(qz)\ \ \text{or}\ \ \psi_{a}(qz)\ \ \text{or}\ \ \varphi_{1}(qz),\\
D(a,q;z)-B(a,q;z)\alpha & = & -\varphi_{a}(z)\ \ \text{or}\ \ -a\psi_{a}(z)\ \ \text{or}\ \ -\varphi_{1}(z),
\end{eqnarray*}
depending on whether $q<a<1$ or $1<a<q^{-1}$ or $a=1$. Equations
(\ref{eq:lim_rel_F_n_all_a}) follow from (\ref{eq:eq:lim_Pnz_over_Qn0})
and (\ref{eq:P_n_and_Q_n_x=00003D0}) while equations (\ref{eq:cases_lim_Qnz_Pnz})
are a direct consequence of (\ref{eq:lim_Qnz_over_Pnz}). \end{proof}

\begin{remark} The limits (\ref{eq:lim_rel_F_n_all_a}) can be proved,
in an alternative way, by applying Darboux's method to the generating
function whose explicit form is given in (\ref{eq:gener_func_F_n}).
According to this method, the leading asymptotic term of $q^{n(n-1)/2}F_{n}(a,q;x)$
is determined by the singularity of the function on the left-hand
side in (\ref{eq:gener_func_F_n}) which is located most closely to
the origin, cf. \cite[Section~8.9]{Olver}. Proceeding this way one
can show that the first limit in (\ref{eq:lim_rel_F_n_all_a}) is
valid even for all $0<a<1$ while the second one is valid for all
$a>1$. Let us also note that the limits established in (\ref{eq:lim_rel_F_n_all_a})
can be interpreted as a $q$-analogue to Hurwitz's limit formula for
the Lommel polynomials. The case $a<1$ has been derived, probably
for the first time, in \cite[Eq. (4.24)]{KoelinkSwarttouw}, see also
\cite[Eq. (3.4)]{KoelinkVanAssche} and \cite[Eq. (2.7)]{Koelink},
while the case $a>1$ has been treated in \cite[Eq. (3.6)]{KoelinkVanAssche}.
\end{remark}

The following formula for the reproducing kernel can be established.

\begin{corollary}
  Suppose $q<a<q^{-1}$. Then
\[
K(u,v)=\frac{a\big(\varphi_{a}(u)\psi_{a}(v)
-\psi_{a}(u)\varphi_{a}(v)\big)}{(1-a)(u-v)}
\]
if $a\neq1$, and
\[
K(u,v)=\frac{\varphi_{1}(u)\big(2q\chi_{1}(v)+v\varphi_{1}'(v)\big)
  -\big(2q\chi_{1}(u)+u\varphi_{1}'(u)\big)\varphi_{1}(v)}{u-v}
\]
if $a=1$.
\end{corollary}

\begin{proof}
  This is a direct consequence of (\ref{eq:repro_kernel}) and
  (\ref{eq:Nevan_ABCD}), (\ref{eq:Nevan_ABCD_a1}).
\end{proof}

\begin{remark}
  \label{rem:sa_extensions}
  In \cite{StampachStovicek_HahnExton}, the self-adjoint extensions of the
  Jacobi matrix $J$, defined in (\ref{eq:alpha_beta}), (\ref{eq:J}), are
  described in detail while addressing only the case $q<a<1$. The self-adjoint
  extensions, called $T(\kappa)$, are parameterized by
  $\kappa\in\mathbb{R}\cup\{\infty\}$, with $\kappa=\infty$ corresponding to
  the Friedrichs extension. The domain $\Dom T(\kappa)\subset\Dom J^{\ast}$ is
  specified by the asymptotic boundary condition: a sequence $f$ from
  $\Dom{}J^{\ast}$ belongs to $\Dom T(\kappa)$ if and only if
  $C_{2}(f)=\kappa\,C_{1}(f)$ where
\[
C_{1}(f)=\lim_{n\to\infty}(-1)^{n}
\left(\frac{a}{q}\right)^{\! n/2}\! f_{n},\ C_{2}(f)
=\lim_{n\to\infty}\left((-1)^{n}f_{n}-C_{1}(f)
\left(\frac{q}{a}\right)^{\! n/2}\right)\!(qa)^{-n/2}
\]
(the limits can be shown to exist). The eigenvalues of $T(\kappa)$ are exactly
the roots of the equation
\[
\kappa\,\,_{1}\phi_{1}(0;qa;q,x)+a\,\,_{1}\phi_{1}(0;qa^{-1};q,a^{-1}x)=0.
\]
On the other hand, consider the self-adjoint extension $T_{t}$ corresponding
the measure of orthogonality $\mu_{t}$, with $t\in\mathbb{R}\cup\{\infty\}$
being a Nevanlinna parameter. The eigenvalues of $T_{t}$ are the
mass points from the support of $\mu_{t}$, i.e. the zeros of the
function
\[
(1-a)\big(B(a,q;x)t-D(a,q;x)\big)=(t+1)a\,_{1}\phi_{1}(0;qa^{-1};q,a^{-1}x)-(t+a)\,_{1}\phi_{1}(0;qa;q,x),
\]
as one infers from (\ref{eq:def_phi_a_psi_a}) and (\ref{eq:Nevan_ABCD}).
Since a self-adjoint extension is unambiguously determined by its
spectrum (see, for instance, proof of Theorem~4.2.4 in \cite{Akhiezer})
one gets the correspondence $\kappa=\Upsilon(t)$ where
\begin{equation}\label{eq:def_Upsilon}
\Upsilon(t)=-\frac{a+t}{1+t}. 
\end{equation}
\end{remark}

\begin{remark}
  Note that if the Krein parametrization (\ref{eq:Krein_abcd}) is used in
  Theorem~\ref{thm:Nevan_ABCD} rather than the Nevanlinna parametrization,
  some parameter functions acquire particularly simple form. For instance,
  $a(z)=\varphi_{a}(-qz)$ and $b(z)=\varphi_{a}(-z)$ provided $q<a<1$. This
  simplification does not apply for functions $c$ and $d$, however.
  
  Being motivated by this observation we propose yet another parametrization
  which is well suited to our problem.  For $1\neq a\in(q,q^{-1})$, let us put
  \begin{eqnarray*}
    && \mathcal{A}(a,q;z) = A(a,q;z)+C(a,q;z),\hskip 16pt
    \mathcal{B}(a,q;z) = -B(a,q;z)-D(a,q;z), \\
    && \mathcal{C}(a,q;z) = aA(a,q;z)+C(a,q;z),\hskip 10pt
    \mathcal{D}(a,q;z) = -B(a,q;z)-a^{-1}D(a,q;z),
  \end{eqnarray*}
  Then $\mathcal{A}(z)=\mathcal{B}(qz)=\varphi_{a}(qz)$ and
  $\mathcal{C}(z)=\mathcal{D}(qz)=\psi_{a}(qz)$. Moreover, we have
  \[
  \frac{A(a,q;z)\varphi(z)-C(a,q;z)}{B(a,q;z)\varphi(z)-D(a,q;z)}
  =-\frac{\mathcal{A}(a,q;z)\Upsilon(\varphi(z))
    +\mathcal{C}(a,q;z)}{\mathcal{B}(a,q;z)\Upsilon(\varphi(z))
    +a\mathcal{D}(a,q;z)}
  \]
  where $\Upsilon$ is defined in (\ref{eq:def_Upsilon}).

  Clearly, $\Upsilon\circ\Upsilon=\mathrm{id}$. Assuming $1<a<q^{-1}$, the
  composition with $\Upsilon\in\mathcal{P}$ induces a one-to-one mapping of
  $\mathcal{P}$ onto itself since the set of Pick functions is closed under
  composition. Thus we obtain a correspondence between solutions
  $\mu_{\omega}$ of the Hamburger moment problem for the polynomial sequence
  $\{P_{n}(a,q;z)\}$ and functions $\omega\in\mathcal{P}\cup\{\infty\}$ which
  is established via the equation
  \[
  \int_{\mathbb{R}}\frac{\mathrm{d}\mu_{\omega}(x)}{z-x}
  =-\frac{\varphi_{a}(qz)\omega(z)+\psi_{a}(qz)}{\varphi_{a}(z)\omega(z)
    +a\psi_{a}(z)}\,,\ \ z\in\mathbb{C}\setminus\mathbb{R}.
  \]
  Similarly, for $q<a<1$, we have $-\Upsilon\in\mathcal{P}$ and an analogous
  correspondence can be established in the form
  \[
  \int_{\mathbb{R}}\frac{\mathrm{d}\mu_{\omega}(x)}{z-x}
  =-\frac{\psi_{a}(qz)\omega(z)+\varphi_{a}(qz)}{a\psi_{a}(z)\omega(z)
    +\varphi_{a}(z)}\,,\ \ z\in\mathbb{C}\setminus\mathbb{R}.
  \]
\end{remark}

\begin{remark}
  There is a close connection between this work and an earlier research
  conducted by Chen and Ismail \cite{ChenIsmail} who studied the indeterminate
  moment problem associated with the orthogonal polynomial sequence
  $\{P_{n}\}_{n=0}^{\infty}$ determined by the recurrence
 \[
  zP_{n}(z)=q^{-n-1}P_{n+1}(z)+q^{-n}P_{n-1}(z) \ \ \text{for}\ \ n\geq1,
 \]
 with $P_{0}(z)=1$, $P_{1}(z)=qz$. These polynomials are a particular case of
 the polynomials $P_{n}(a,q;x)$ defined in (\ref{eq:P_n_rel_F_n}). Indeed, for
 all $n\in\mathbb{Z}_{+}$,
 \[
  P_{n}(z)=P_{n}(-1,q;i\sqrt{q}z).
 \]
 For this particular choice of parameters one can deduce from
 (\ref{eq:gener_func_F_n}) that a generating function for the polynomials
 $\{P_{n}(z)\}$ can be expressed in the form
 \[
  \sum_{n=0}^{\infty}P_{n}(z)t^{n}=\sum_{k=0}^{\infty}\frac{q^{k(k+1)/2}\left(zt\right)^{k}}{(-qt^{2};q^{2})_{k+1}}
 \]
 coinciding with a formula obtained by Chen and Ismail, see
 \cite[Theorem~3.1]{ChenIsmail}. Note, however, that $a$, being specialized to
 $-1$, is out of the range considered in the present paper.

 With the knowledge of the generating function one can proceed similarly as in
 Theorem~\ref{thm:Nevan_ABCD} reproducing this way a result concerned with the
 Nevanlinna parametrization which has been obtained by Chen and Ismail in
 \cite[Theorem~3.2]{ChenIsmail}. Let us denote by $\tilde{A}$, $\tilde{B}$,
 $\tilde{C}$, and $\tilde{D}$ the Nevanlinna functions associated with the
 polynomials $P_{n}(z)$. Using the notation introduced in
 (\ref{eq:Nevan_ABCD}) we have
\begin{eqnarray*}
&&  \tilde{A}(z)=iq^{1/2}A\!\left(-1,q;iq^{1/2}z\right), \quad
 \tilde{B}(z)=B\!\left(-1,q;iq^{1/2}z\right), \\
&&  \tilde{C}(z)=C\!\left(-1,q;iq^{1/2}z\right), \quad\hskip 22pt
 \tilde{D}(z)=-iq^{-1/2}D\!\left(-1,q;iq^{1/2}z\right).
\end{eqnarray*}
Let us remark that $\varphi_{-1}(z)=\psi_{-1}(-z)$. This makes it possible to
express $\tilde{A}$, $\tilde{B}$, $\tilde{C}$, and $\tilde{D}$ in a
comparatively simple form. For instance, a straightforward computation yields
 \[
 \tilde{A}(z)
 =\sum_{n=0}^{\infty}\frac{(-1)^{n}q^{2(n+1)^{2}}}{(q^{2};q^{2})_{2n+1}}\,z^{2n+1}.
 \]
\end{remark}

\subsection{Measures of orthogonality}

With the explicit knowledge of the Nevanlinna parametrization established
in Theorem~\ref{thm:Nevan_ABCD} it is straightforward to describe
all N-extremal solutions.

\begin{proposition}
\label{prop:Nex_OG_measure}
Let $1\neq a\in(q,q^{-1})$.  Then all N-extremal measures
$\mu_{t}=\mu_{t}(a,q)$, $t\in\mathbb{R}\cup\{\infty\}$, are of the form
\begin{equation}
  \mu_{t}=\sum_{x\in\mathfrak{Z}_{t}}
  \rho(x)\,\delta_{x}\ \ \text{where}\ \ \frac{1}{\rho(x)}
  =\frac{a}{1-a}\big(\psi_{a}(x)\varphi_{a}'(x)-\varphi_{a}(x)\psi_{a}'(x)\big),
  \label{eq:Nex_OG_measure}
\end{equation}
\[
\mathfrak{Z}_{t}=\mathfrak{Z}_{t}(a,q)
=\{x\in\mathbb{R};\, a(t+1)\psi_{a}(x)-(t+a)\varphi_{a}(x)=0\},
\]
and $\delta_{x}$ stands for the Dirac measure supported on $\{x\}$.

For $a=1$, all N-extremal measures $\mu_{t}=\mu_{t}(1,q)$ are of
the form $\mu_{t}=\sum_{x\in\mathfrak{Y}_{t}}\rho(x)\,\delta_{x}$
where
\[
\frac{1}{\rho(x)}=2q\big(\varphi_{1}'(x)\chi_{1}(x)
-\varphi_{1}(x)\chi_{1}'(x)\big)+x\big(\varphi_{1}'(x)\big)^{2}
-\varphi_{1}(x)\varphi_{1}'(x)-x\varphi_{1}(x)\varphi_{1}''(x).
\]
and
\[
\mathfrak{Y}_{t}=\mathfrak{Y}_{t}(q)
=\{x\in\mathbb{R};\,2q(t+1)\chi_{1}(x)+(t+1)x\varphi_{1}'(x)
-t\varphi_{1}(x)=0\}.
\]
\label{prop:N_extr_a_neq_1}
\end{proposition}

\begin{proof}
  Referring to general formulas (\ref{eq:mu_t}) and (\ref{eq:Zero_t}),
  (\ref{eq:rho}), it suffices to apply Theorem~\ref{thm:Nevan_ABCD}.
\end{proof}

%%%%%%%%%%%%%%%%%%%%%%%%%%%%%%%%%%%%%%%%%%%%%%%%%%%%%%%%%%%%%%%%%%%%%%%
The formula for the orthogonality measure $\mu_{t}$, as given in
(\ref{eq:Nex_OG_measure}), is far of being explicit. To elucidate its
structure one can attempt at least to provide the asymptotics for the mass
points $x\in\mathfrak{Z}_{t}$ and the weights $\rho(x)$ for $x$ large. This is
actually possible owing to the fact that a good deal of attention has been
paid by various authors to the study of the asymptotic behavior of the roots
of the $q$-Bessel functions. We have postponed to Appendix a summary of basic
facts and references on this subject that we need for our purposes. Moreover,
these facts are completed therein with some additional details. In the
asymptotic analysis to follow we restrict our attention to the range of
parameters $0<q<a<1$ only.

Suppose $t\in\mathbb{R}\cup\{\infty\}$. As pointed out in
Remark~\ref{rem:sa_extensions}, the mass points of the measure of
orthogonality $\mu_{t}$ are, at the same time, eigenvalues of the
corresponding self-adjoint extension $T_{t}$ of the matrix (Jacobi) operator
$J$ introduced in (\ref{eq:J}). The eigenvalues are solutions to the
characteristic equation $\Phi_{t}(x)=0$ where
\begin{eqnarray}
  \Phi_{t}(x) & = & a\,(t+1)\,\psi_{a}(x)-(t+a)\,\varphi_{a}(x)\nonumber \\
  & = & a\,(t+1)\,_{1}
  \phi_{1}\!\left(0;\frac{q}{a};q,\frac{x}{a}\right)-(t+a)\,_{1}
  \phi_{1}(0;aq;q,x),
  \label{eq:chareq_Phi}
\end{eqnarray}
they are all real and simple, and their only accumulation point is
$+\infty$. If $x$ is an eigenvalue of $T_{t}$ then the root $x$ of
$\Phi_{t}(x)$ is also simple. Actually, this is a standard fact following from
the discrete Green-like formula implying that the square norm of the
corresponding eigenvector is proportional to $\Phi_{t}'(x)$ (see
\cite[\S~I.2.1]{Akhiezer} or, for instance,
\cite[Lemma~4]{StampachStovicek_CoulombOP}).

Let us order the mass points of the measure $\mu_{t}$, i.e. the roots of
$\Phi_{t}(x)$, increasingly,
$\xi_{0}^{(t)}<\xi_{1}^{(t)}<\xi_{2}^{(t)}<\ldots$. More is known about the
properties of these roots, see \S~2.4 in \cite{StampachStovicek_HahnExton} for
additional details. Remember $T_{-1}$ coincides with the Friedrichs extension
of $J$. For any fixed $m\in\mathbb{Z}_{+}$ and using the substitution
$t=-(\kappa+a)/(\kappa+1)$, one knows that thus obtained composed function
$\xi_{m}^{(t)}$ is strictly monotonic increasing in the parameter
$\kappa\in\mathbb{R}$.  Furthermore, for any $t\in\mathbb{R}\cup\{\infty\}$,
$t\neq-1$, we have the inequalities
\[
\xi_{0}^{(t)}<\xi_{0}^{(-1)}\text{ }\text{and}\text{ }
\xi_{m-1}^{(-1)}<\xi_{m}^{(t)}<\xi_{m}^{(-1)}\text{ }\text{for}\ m\geq1.
\]

Moreover, Proposition~\ref{thm:asympt_roots_phi} from Appendix tells us that
\begin{equation}
  \xi_{m}^{(-1)}=q^{-m}+O(a^{m}q^{m(m+1)})\text{ }\text{and}\text{ }
  \xi_{m}^{(-a)}=aq^{-m}+O(a^{-m}q^{m(m+1)})\text{ }\text{as}\ m\to\infty.
  \label{eq:asym_xi_part}
\end{equation}
For a general $t$ we have the following result.

\begin{proposition}
\label{thm:asym_roots_t}
Suppose $0<q<a<1$ and let $t\in\mathbb{R}\cup\{\infty\}$, $t\neq-1$. Then
\begin{equation}
  \xi_{m}^{(t)}=aq^{-m}\!\left(1-\frac{(1-a)(t+a)}{t+1}
    \frac{(q/a;q)_{\infty}^{\,2}}{(q;q)_{\infty}^{\,2}}\, a^{m-1}+O(a^{2m})\right)
  \label{eq:asym_roots_t}
\end{equation}
and
\begin{equation}
  \frac{1}{\rho(\xi_{m}^{(t)})}
  =(q;q)_{\infty}^{\,2}\, a^{m}q^{-m^{2}}\left(1+O(ma^{m})\right)
  \label{eq:asym_weights_t}
\end{equation}
as $m\to\infty$.
\end{proposition}

\begin{proof}
  We seek a root $\xi$ of the characteristic equation in vicinity of $aq^{-m}$
  while writing $\xi=aq^{-m}+\epsilon$, with $|q^{m}\epsilon|\ll1$, and
  assuming $m$ to be sufficiently large.  Referring to
  Theorem~\ref{thm:asympt_q-confhyp} one readily concludes that
\[
\frac{A(x)}{A(x/a)}=D(x)\exp\!\left(-\frac{\ln(a)}{\ln(q)}\ln(x)\right)
\text{where}\ D(x)=O(1)\ \text{as}\ x\to+\infty,
\]
and therefore, in the considered asymptotic domain, the characteristic
equation can be given the form
\[
\sin\!\left(\frac{\pi\ln(x/a)}{\ln(q)}\right)
=O\!\left(x^{-\ln(a)/\ln(q)}\right).
\]
Hence, for all sufficiently large $m$, a solution $\xi=\xi(m)$ is sure to
exist such that $\xi(m)=aq^{-m}\left(1+O(a^{m})\right)$ as $m\to\infty$. To
get the second asymptotic term explicitly note that, in view of
(\ref{eq:daalhuis}),
\begin{eqnarray*}
  \frac{A(aq^{-m})}{A(q^{-m})}
  & = & a^{m+1/2}\exp\!\left(-\frac{\ln^{2}(a)}{2\ln(q)}\right)
  \frac{\left|\left(\tilde{q}e^{-2i\beta(a)};
        \tilde{q}\right)_{\!\infty}\right|^{2}}
  {(\tilde{q};\tilde{q})_{\infty}^{\,2}}\\
  \noalign{\smallskip}
  & = & -\frac{(a;q)_{\infty}(q/a;q)_{\infty}}{(q;q)_{\infty}^{\,2}}
  \frac{\pi a^{m}}{\ln(q)\sin(\beta(a))}.
\end{eqnarray*}
Now it is somewhat tedious but straightforward to apply
Theorem~\ref{thm:asympt_q-confhyp} to (\ref{eq:chareq_Phi}) in order to derive
that
\begin{equation}
  \epsilon(m)=-aq^{-m}\,\frac{(1-a)(t+a)}{t+1}
  \frac{(q/a;q)_{\infty}^{\,2}}{(q;q)_{\infty}^{\,2}}\, a^{m-1}
  \left(1+O(a^{m})\right).
  \label{eq:asymp_eps_m}
\end{equation}
Hence the sequence of roots $\{\xi(m)\}$ we have found meets the asymptotic
behavior as claimed in the proposition. To conclude the proof it suffices to
observe that (\ref{eq:asymp_eps_m}) along with (\ref{eq:asym_xi_part})
guarantee $\xi(m)$ to be sufficiently close to $\xi_{m}^{(-a)}$ and,
consequently, $\xi_{m-1}^{(-1)}<\xi(m)<\xi_{m}^{(-1)}$, implying that
$\xi(m)=\xi_{m}^{(t)}$ is actually the $m$th root of $\Phi_{t}(x)$.

As far as the weights are concerned, sticking to notation (\ref{eq:Az}) and
making use of (\ref{eq:asym_roots_t}) it is straightforward to derive that
\[
A(\xi_{m}^{(t)})=A(a)a^{m}q^{-(m+1)m/2}\left(1+O(ma^{m})\right).
\]
Clearly,
\[
\sin\!\left(\beta(\xi_{m}^{(t)})\right)
=(-1)^{m}\sin(\beta(a))\left(1+O(a^{m})\right).
\]
Taking into account (\ref{eq:daalhuis}) and Theorem~\ref{thm:asympt_q-confhyp}
one finds that
\[
\varphi_{a}(\xi_{m}^{(t)})=(-1)^{m}(q/a;q)_{\infty}(1-a)a^{m}q^{-(m+1)m/2}
\left(1+O(ma^{m})\right).
\]
To analyze
\[
\psi_{a}'(\xi_{m}^{(t)})
=\frac{1}{a}\,\frac{\partial\,_{1}\phi_{1}(0;q/a;q,z)}
{\partial z}\Bigg|_{z=\xi_{m}^{(t)}/a}
\]
for $m$ large one can make use of Theorem~\ref{thm:asymp_der_q-confhyp} along
with (\ref{eq:ident_qPochhammer}) to obtain
\[
\psi_{a}'(\xi_{m}^{(t)})
=(-1)^{m+1}\frac{(q;q)_{\infty}^{\,2}}
{a\,(q/a;q)_{\infty}}\, q^{-m(m-1)/2}\left(1+O(ma^{m})\right).
\]
Similarly one finds that $\varphi_{a}'(\xi_{m}^{(t)})=O(ma^{m}q^{-m(m-1)/2})$.
Consequently, in view of (\ref{eq:Nex_OG_measure}),
\begin{eqnarray*}
  \frac{1}{\rho(\xi_{m}^{(t)})}
  & = & \frac{a\,\varphi_{a}(\xi_{m}^{(t)})}{1-a}
  \left(\frac{t+a}{a\,(t+1)}\,\varphi_{a}'(\xi_{m}^{(t)})
    -\psi_{a}'(\xi_{m}^{(t)})\right)\\
  & = & -\frac{a}{1-a}\,\varphi_{a}(\xi_{m}^{(t)})\,
  \psi_{a}'(\xi_{m}^{(t)})\left(1+O(ma^{m})\right).
\end{eqnarray*}
Relation (\ref{eq:asym_weights_t}) readily follows.
\end{proof}
%%%%%%%%%%%%%%%%%%%%%%%%%%%%%%%%%%%%%%%%%%%%%%%%%%%%%%%%%%%%%%%%%%%%%%%

In what follows we again admit any value of $a$ lying between $q$ and
$q^{-1}$.

\begin{lemma}
\label{prop:AD-BC_in_phipsi}
With the notation introduced in (\ref{eq:def_phi_a_psi_a}) it holds true that
\begin{equation}
  \varphi_{a}(z)\psi_{a}(qz)-a\psi_{a}(z)\varphi_{a}(qz)
  =1-a\ \ \mbox{ if}\ \ a\neq1,\label{eq:determ_Nevan_a_neq_1}
\end{equation}
and
\[
2q\big(\varphi_{1}(z)\chi_{1}(qz)-\chi_{1}(z)\varphi_{1}(qz)\big)
+z\big(q\varphi_{1}(z)\varphi_{1}'(qz)
-\varphi_{1}'(z)\varphi_{1}(qz)\big)+\varphi_{1}(z)\varphi_{1}(qz)=1,
\]
for all $z\in\mathbb{C}$.
\end{lemma}

\begin{proof}
  These identities follow from (\ref{eq:AD-BC=00003D1}) and, again, from
  Theorem~\ref{thm:Nevan_ABCD}.
\end{proof}

Let us examine a bit more closely two particular N-extremal measures $\mu_{t}$
described in Proposition~\ref{prop:Nex_OG_measure}, with $t=-1$ and
$t=-a$. They correspond to the distinguished cases $t=\alpha$ if $a\in(q,1)$
or $a\in(1,q^{-1})$, respectively (cf. (\ref{eq:quant_alpha})). As already
mentioned, if $t=\alpha$ then the corresponding self-adjoint extension of the
underlying Jacobi matrix is the Friedrichs extension, and the measure
$\mu_{t}$ is necessarily a Stieltjes measure. In the case $t=-1$ the
orthogonality relation for the orthonormal polynomials $P_{n}(a,q;x)$ reads
\begin{equation}
  -\sum_{k=0}^{\infty}\frac{\varphi_{a}(q\xi_{k})}
  {\varphi_{a}'(\xi_{k})}\, P_{n}(a,q;\xi_{k})P_{m}(a,q;\xi_{k})
  =\delta_{mn}
  \label{eq:OGrel_t_eq_-1}
\end{equation}
where $\{\xi_{k};\, k\in\mathbb{Z}_+\}$ are the zeros of the function
$\varphi_{a}$. Actually, from (\ref{eq:Nex_OG_measure}) and
(\ref{eq:determ_Nevan_a_neq_1}) one infers that
$\rho(\xi_k)=-\varphi_{a}(q\xi_{k})/\varphi_{a}'(\xi_{k})$ if
$\varphi_{a}(\xi_{k})=0$. Similarly, the same orthogonality relation for
$t=-a$ reads
\begin{equation}
  -\frac{1}{a}\sum_{k=0}^{\infty}
  \frac{\psi_{a}(q\eta_{k})}{\psi_{a}'(\eta_{k})}\, P_{n}(a,q;
  \eta_{k})P_{m}(a,q;\eta_{k})=\delta_{mn}
  \label{eq:OGrel_t_eq_-a}
\end{equation}
where $\{\eta_{k};\, k\in\mathbb{Z}_+\}$ are the zeros of the function
$\psi_{a}$. One can refer once more to \cite[Thm.~2.2,
Rem.~2.3]{AnnabyMansour} where it is shown (if rewritten in our notation) that
 \[
 \xi_{k}=q^{-k}\left(1+O(q^{k})\right)\text{ }\text{and}\text{ }\text{ } 
\eta_{k}=aq^{-k}\left(1+O(q^{k})\right)\ \ \ \text{as }\ \ k\to\infty.
 \]

\begin{remark}
  The orthogonality relation (\ref{eq:OGrel_t_eq_-1}) has been derived already
  in \cite[Theorem 3.6.]{Koelink}. This is the unique orthogonality relation
  for the polynomials $P_{n}(a,q;x)$ if $a\in(0,q]$ (the determinate case),
  and an example of an N-extremal orthogonality relation if
  $a\in(q,1)$. Similarly, (\ref{eq:OGrel_t_eq_-a}) is the unique orthogonality
  relation if $a\geq q^{-1}$. Of course, (\ref{eq:OGrel_t_eq_-1}) and
  (\ref{eq:OGrel_t_eq_-a}) coincide for $a=1$.
\end{remark}

\begin{remark}
  In \cite[Section~1]{BergChristensen}, an explicit expression has been found
  for the measures of orthogonality $\mu_{\varphi}$ corresponding to constant
  Pick functions $\varphi(z)=\beta+i\gamma$, with $\beta\in\mathbb{R}$ and
  $\gamma>0$. Let us call these measures
  $\mu_{\beta,\gamma}=\mu_{\beta,\gamma}(a,q)$. It turns out that
  $\mu_{\beta,\gamma}$ is an absolutely continuous measure supported on
  $\mathbb{R}$ with the density
\[
\frac{\mbox{d}\mu_{\beta,\gamma}}{\mbox{d}x}
=\frac{\gamma}{\pi}\big(\big(\beta B(a,q;x)-D(a,q;x)\big)^{2}
+\gamma^{2}B(a,q;x)^{2}\big)^{-1}.
\]
In our case, referring to (\ref{eq:Nevan_ABCD}), (\ref{eq:Nevan_ABCD_a1}), we
get the probability density
\[
\frac{\mbox{d}\mu_{\beta,\gamma}}{\mbox{d}x}
=\frac{\gamma\,(1-a)^{2}}{\pi\!\left(\big((\beta+1)a\psi_{a}(x)
-(\beta+a)\varphi_{a}(x)\big)^{2}+\gamma^{2}\big(a\psi_{a}(x)
-\varphi_{a}(x)\big)^{2}\right)}\,,
\]
provided $1\neq a\in(q,q^{-1})$, and
\begin{eqnarray*}
\hskip-1.8em 
&  & \frac{\mbox{d}\mu_{\beta,\gamma}}{\mbox{d}x}\,=\,\frac{\gamma}{\pi}\\
\hskip-1.8em 
&  & \times\!\left(\!\big(2q(\beta+1)\chi_{1}(x)-\beta\varphi_{1}(x)
+(\beta+1)x\varphi_{1}'(x)\big)^{2}+\gamma^{2}\big(2q\chi_{1}(x)
-\varphi_{1}(x)+x\varphi_{1}'(x)\big)^{2}\right)^{\!-1}\!\!\!,
\end{eqnarray*}
provided $a=1$. Letting $\beta=-1$ or $\beta=-a$ and $\gamma>0$ arbitrary, one
obtains comparatively simple and nice orthogonality relations for the
polynomials $P_{n}(a,q;x)$, namely
\[
\int_{\mathbb{R}}\frac{P_{m}(a,q;x)P_{n}(a,q;x)}
{\gamma\big(a\psi_{a}(x)-\varphi_{a}(x)\big)^{2}
+\gamma^{-1}(a-1)^{2}\varphi_{a}(x)^{2}}\,\mbox{d}x
=\frac{\pi}{(a-1)^{2}}\,\delta_{mn}
\]
and
\[
\int_{\mathbb{R}}\frac{P_{m}(a,q;x)P_{n}(a,q;x)}
{\gamma\big(a\psi_{a}(x)-\varphi_{a}(x)\big)^{2}
+\gamma^{-1}(a-1)^{2}a^{2}\psi_{a}(x)^{2}}\,\mbox{d}x
=\frac{\pi}{(a-1)^{2}}\,\delta_{mn},
\]
valid for all $m,n\in\mathbb{Z}_{+}$ and $a\in(q,q^{-1})$, $a\neq1$.  If
$a=1$, a similar orthogonality relation takes the form
\[
\int_{\mathbb{R}}\frac{P_{m}(1,q;x)P_{n}(1,q;x)}
{\gamma\big(2q\chi_{1}(x)+x\varphi_{1}'(x)-\varphi_{1}(x)\big)^{2}
+\gamma^{-1}\varphi_{1}(x)^{2}}\,\mbox{d}x=\pi\delta_{mn}.
\]
\end{remark}

\section{The moment sequence}

\subsection{Passing to the determinate case}

Let $\mu$ be any measure of orthogonality for the orthonormal polynomials
$P_{n}(a,q;x)$ introduced in (\ref{eq:P_n_rel_F_n}). Denote by
\[
m_{n}(a,q)=\int_{\mathbb{R}}x^{n}\,\mbox{d}\mu(x),\ \ n\in\mathbb{Z}_{+},
\]
the corresponding moment sequence. It is clear from Favard's theorem,
however, that the moments do not depend on the particular choice of
the measure of orthogonality. It is even known that
\begin{equation}
m_{n}(a,q)=\langle e_{0},J(a,q)^{n}e_{0}\rangle,\ \ n\in\mathbb{Z}_{+},\label{eq:momentseq_J}
\end{equation}
where $J(a,q)$ is the Jacobi matrix defined in (\ref{eq:alpha_beta}),
(\ref{eq:J}), and $e_{0}$ is the first vector of the canonical basis
in $\ell^{2}(\mathbb{Z}_{+})$. Whence $m_{n}(a,q)$ is a polynomial
in $a$ and $q^{-1}$. Consequently, in order to compute the moments
one can admit a wider range of parameters than that we were using
up to now, namely $0<q<1$ and $q<a<q^{-1}$. This observation can
be of particular importance for the parameter $q$ since the properties
of the matrix operator $J(a,q)$ would change dramatically if $q$
was allowed to take values $q>1$. We wish to stick, however, to the
widely used convention according to which the modulus of $q$ is smaller
than $1$. This is why we replace the symbol $q$ by $p$ in this
section whenever this restriction is relaxed. Concerning the parameter
$a$, it is always supposed to be positive.

Put, for $p>0$ and $a>0$,
\begin{equation}
\omega_{n}(a,p)=\sum_{k=0}^{n}\left[\begin{array}{c}
n\\
k
\end{array}\right]_{p}p^{-k(n-k)}a^{k},\ \text{ }n\in\mathbb{Z}_{+}.\label{eq:def_omega_n}
\end{equation}
The meaning of the $q$-binomial ceofficient in (\ref{eq:def_omega_n})
is the standard one, cf. \cite[Eq.~(I.39)]{GasperRahman}. Let us
remark that $\omega_{n}(a,p)$ can be expressed in terms of the continuous
$q$-Hermite polynomials $H_{n}(x;q)$, namely

\begin{equation}
\omega_{n}(a,p)=\,_{2}\phi_{0}(p^{n},0;\,;p^{-1},p^{-n}a)=a^{n/2}\, H_{n}\Big(\,\frac{1}{2}\left(a^{1/2}+a^{-1/2}\right);p^{-1}\Big),\label{eq:omega_Hn}
\end{equation}
see \cite{KoekoekLeskySwarttouw}.

As before, the monic polynomials $F_{n}(a,p;x)$ are generated by
the recurrence (\ref{eq:recur_def_F_n}), with $F_{-1}(a,p;x)=0$
and $F_{0}(a,p;x)=1$ (writing $p$ instead of $q$). The following
proposition is due to Van~Assche and is contained in \cite[Theorem~2]{VanAssche}.

\begin{proposition} \label{prop:lim_VanAssche} For $p>1$ and $x\neq0$
one has
\[
\lim_{n\to\infty}\, x^{-n}F_{n}(a,p;x)=\sum_{k=0}^{\infty}\frac{\omega_{k}(a,p)}{(p;p)_{k}}\left(\frac{p}{x}\right)^{k}.
\]
\end{proposition}

Note that if $p>1$ then the Jacobi matrix $J(a,p)$ represents a
compact (even trace class) operator on $\ell^{2}(\mathbb{Z}_{+})$.
In particular, this implies that the Hamburger moment problem is determinate.
Several additional useful facts are known in this case which we summarize
in the following remark.

\begin{remark} \label{rem:CoulombOP_results} In \cite[Section~3]{StampachStovicek_CoulombOP}
it is noted that if $\{\beta_{n}\}_{n=0}^{\infty}$ is a real sequence
belonging to $\ell^{1}(\mathbb{Z}_{+})$, $\{\alpha_{n}\}_{n=0}^{\infty}$
is a positive sequence belonging to $\ell^{2}(\mathbb{Z}_{+})$ and
$\{F_{n}(x)\}_{n=0}^{\infty}$ is a sequence of monic polynomials
defined by the recurrence
\[
F_{n+1}(x)=(x-\beta_{n})F_{n}(x)-\alpha_{n-1}^{\,2}F_{n-1}(x),\ \ n\geq0,
\]
with $F_{0}(x)=1$ and (conventionally) $F_{-1}(x)=0$, then
\begin{equation}
\lim_{n\to\infty}\, x^{-n}F_{n}(x)=\mathcal{G}(x^{-1})\ \ \text{for}\ \ x\neq0\label{eq:lim_Fn_calG}
\end{equation}
where $\mathcal{G}(z)$ is an entire function. Moreover, let $\mu$
be the (necessarily unique) measure of orthogonality for the sequence
of polynomials $\{F_{n}(x)\}$. Then the Stieltjes transform of $\mu$
reads
\begin{equation}
\int_{\mathbb{R}}\frac{\mbox{d}\mu(x)}{1-zx}=\frac{\tilde{\mathcal{G}}(z)}{\mathcal{G}(z)}\label{eq:Stieltjes_mu_general}
\end{equation}
where $\tilde{\mathcal{G}}(z)$ is an entire function associated in
an analogous manner with the shifted sequences $\{\tilde{\alpha}_{n}=\alpha_{n+1}\}_{n=0}^{\infty}$,
$\{\tilde{\beta}_{n}=\beta_{n+1}\}_{n=0}^{\infty}$. \end{remark}

\begin{theorem} Let $p>1$ and $x\neq0$. Then
\[
\lim_{n\to\infty}\, x^{-n}F_{n}(a,p;x)=\mathcal{G}(x^{-1})
\]
where
\begin{equation}
\mathcal{G}(z)=(z;p^{-1})_{\infty}\,\,_{1}\phi_{1}(0;z;p^{-1},az)=\sum_{k=0}^{\infty}\frac{\omega_{k}(a,p)}{(p;p)_{k}}\,(pz)^{k}\label{eq:calG_phi11_series}
\end{equation}
is an entire function obeying the second-order $q$-difference equation
\begin{equation}
\mathcal{G}(z)-\big(1-(a+1)z\big)\mathcal{G}(p^{-1}z)+ap^{-1}z^{2}\mathcal{G}(p^{-2}z)=0.\label{eq:qdiff_G_J}
\end{equation}
The Stieltjes transform of the (unique) measure of orthogonality $\mu$
for the sequence of orthogonal polynomials $\{F_{n}(a,p;x)\}$ is
given by the formula
\begin{equation}
\int_{\mathbb{R}}\frac{\mbox{d}\mu(x)}{1-zx}=\frac{\mathcal{G}(p^{-1}z)}{\mathcal{G}(z)}\,.\label{eq:Stieltjes_mu}
\end{equation}
\end{theorem}

\begin{proof} In view of Proposition~\ref{prop:lim_VanAssche},
in order to show (\ref{eq:calG_phi11_series}) it suffices to verify
only the second equality. But this equality follows from the definition
of the basic hypergeometric series and from the well known identity
\cite[Eq. (II.2)]{GasperRahman}
\[
(z;p^{-1})_{\infty}=\sum_{n=0}^{\infty}\frac{(pz)^{n}}{(p;p)_{n}}\,.
\]

Using the power series expansion of $\mathcal{G}(z)$ established
in (\ref{eq:calG_phi11_series}) one finds that (\ref{eq:qdiff_G_J})
is equivalent to
\[
\omega_{k}-(a+1)\omega_{k-1}+a\,(1-p^{-k+1})\omega_{k-2}=0\ \ \text{for}\ \ k\geq2
\]
and $\omega_{1}-(a+1)\omega_{0}=0$. This is true, indeed, if we take
into account (\ref{eq:omega_Hn}) and the recurrence relation for
the continuous $q$-Hermite polynomials \cite[Eq. (14.26.3)]{KoekoekLeskySwarttouw}
\[
2xH_{k}(x;q)=H_{k+1}(x;q)+(1-q^{k})H_{k-1}(x;q).
\]

Recalling once more (\ref{eq:recur_def_F_n}), the polynomials $F_{n}(a,p;x)$
solve the recurrence relation
\begin{equation}
  u_{n+1}=\big(x-(a+1)p^{-n}\big)u_{n}-ap^{-2n+1}u_{n-1}
  \label{eq:recurr_F}
\end{equation}
while the polynomials $\tilde{F}_{n}(a,p;x):=p^{-n}F_{n}(a,p;px)$ obviously
obey the recurrence
\begin{equation}
  \tilde{u}_{n+1}
  =\big(x-(a+1)p^{-n-1}\big)\tilde{u}_{n}-ap^{-2n-1}\tilde{u}_{n-1}.
  \label{eq:recurr_Ftilde}
\end{equation}
Comparing these two equations one observes that (\ref{eq:recurr_Ftilde}) is
obtained from (\ref{eq:recurr_F}) just by shifting the index.  In other words,
the sequences of monic polynomials $\{\tilde{F}_{n}(a,p;x)\}$ and
$\{F_{n}(a,p;x)\}$ are generated by the same recurrence relation, but the
index has to be shifted in the latter case. Hence, referring to
Remark~\ref{rem:CoulombOP_results} and equation (\ref{eq:lim_Fn_calG}), one
can compute
\[
\tilde{\mathcal{G}}(x^{-1})
=\lim_{n\to\infty}\, x^{-n}\tilde{F}_{n}(a,p;x)
=\lim_{n\to\infty}\,(px)^{-n}F_{n}(a,p;px)=\mathcal{G}(p^{-1}x^{-1}).
\]
Thus $\tilde{\mathcal{G}}(z)=\mathcal{G}(p^{-1}z)$ and (\ref{eq:Stieltjes_mu})
is a particular case of (\ref{eq:Stieltjes_mu_general}). \end{proof}

\subsection{Recurrence relations and asymptotic behavior}

From (\ref{eq:momentseq_J}) it is seen that $m_{n}(a,p)\leq\|J(a,p)\|^{n}$.
Moreover, from (\ref{eq:Stieltjes_mu}) one deduces that
\begin{equation}
\sum_{n=0}^{\infty}m_{n}(a,p)z^{n}=\frac{\mathcal{G}(p^{-1}z)}{\mathcal{G}(z)}\,,\label{eq:moment_powerser_eq_ratioG}
\end{equation}
and the series is clearly convergent if $p>1$ and $|z|<\|J(a,p)\|^{-1}$.

\begin{remark} Any explicit formula for monic polynomials $F_{n}(x)$,
$n\in\mathbb{Z}_{+}$, which are members of a sequence of orthogonal
polynomials with a measure of orthogonality $\mu$, automatically
implies a linear recursion for the corresponding moments. In fact,
$F_{0}(x)=1$ and so, by orthogonality, $\int_{\mathbb{R}}F_{n}(x)\mbox{d}\mu(x)=0$
for $n\geq1$. Particularly, in our case, formula (\ref{eq:F_n_explicit})
implies the relation
\[
\sum_{j=0}^{n}\frac{(-1)^{j}q^{(j-1)j/2}}{(q;q)_{j}{}^{2}}\left(\sum_{k=0}^{n-j}(q^{k+1};q)_{j}(q^{n-j-k+1};q)_{j}\, a^{k}\right)\! m_{j}(a,q)=0\ \ \text{for}\ \ n\geq1.
\]
\end{remark}

Further we derive two more recursions for the moments, a linear and
a quadratic one.

\begin{proposition} The moment sequence $\{m_{n}(a,q)\}$ solves
the equations $m_{0}(a,q)=1$ and
\begin{equation}
m_{n}(a,q)=\frac{\omega_{n}(a,q)}{(q;q)_{n-1}}-\sum_{k=1}^{n-1}\frac{q^{k}\omega_{k}(a,q)}{(q;q)_{k}}\, m_{n-k}(a,q),\ \ n\in\mathbb{N}.\label{eq:lin_recur_moments}
\end{equation}
\end{proposition}

\begin{proof} Equations (\ref{eq:moment_powerser_eq_ratioG}) and
(\ref{eq:calG_phi11_series}) imply that
\[
\sum_{m=0}^{\infty}\frac{p^{m}\omega_{m}(a,p)}{(p;p)_{m}}\, z^{m}\,\sum_{n=0}^{\infty}m_{n}(a,p)z^{n}=\sum_{m=0}^{\infty}\frac{\omega_{m}(a,p)}{(p;p)_{m}}\, z^{m}
\]
holds for $p>1$ and $z$ from a neighborhood of $0$. Equating the
coefficients of equal powers of $z$ one finds that (\ref{eq:lin_recur_moments})
holds true for $q=p>1$. But for the both sides are rational functions
in $q$ the equation remains valid also for $0<q<1$. \end{proof}

\begin{proposition} The moment sequence $\{m_{n}(a,q)\}$ solves
the equations $m_{0}(a;q)=1$ and
\begin{equation}
m_{n+1}(a,q)=(a+1)\, m_{n}(a,q)+a\sum_{k=0}^{n-1}q^{-k-1}m_{k}(a,q)\, m_{n-k-1}(a,q),\ \ n\in\mathbb{Z}_{+}.\label{eq:quadr_recur_moments}
\end{equation}
\end{proposition}

\begin{proof} Equation (\ref{eq:qdiff_G_J}) can be rewritten as
\[
\frac{\mathcal{G}(p^{-1}z)}{\mathcal{G}(z)}\!\left(1-(a+1)z-ap^{-1}z^{2}\,\frac{\mathcal{G}(p^{-2}z)}{\mathcal{G}(p^{-1}z)}\right)=1
\]
and holds true for $p>1$ and $z$ from a neighborhood of the origin.
Substituting the power series expansion (\ref{eq:moment_powerser_eq_ratioG})
one has
\[
\left(1-(a+1)z-ap^{-1}z^{2}\sum_{n=0}^{\infty}m_{n}(a,p)\, p^{-n}z^{n}\right)\!\sum_{n=0}^{\infty}m_{n}(a,p)z^{n}=1.
\]
Equating the coefficients of equal powers of $z$ one concludes that
(\ref{eq:quadr_recur_moments}) holds for $q=p>1$. For the both sides
are polynomials in $q^{-1}$ the equation is valid for $0<q<1$ as
well. \end{proof}

Our final task is to provide estimates bringing some insight into
the asymptotic behavior of the moments for large powers. We still
assume that $0<q<1$ and $a>0$. On the other hand, $a$ is not required
to be restricted to the interval $q<a<q^{-1}$. Let us note that it
has been shown in \cite[Lemma~4.9.1]{BergValent} that
\begin{equation}
a^{n/2}q^{-n(n-1)/4}\leq\omega_{n}(a,q)\leq(1+a)^{n}q^{-n^{2}/4},\ \ n\in\mathbb{Z}_{+}.\label{eq:estim_omega}
\end{equation}

\begin{proposition} Let $a>0$. The moments $m_{n}(a,q)$ obey the
inequalities
\begin{equation}
m_{n}(a,q)\leq\frac{(1+a)^{n}}{(q;q)_{n-1}}\, q^{-n^{2}/4},\ n\in\mathbb{Z}_{+},\label{eq:upper_estim_moment}
\end{equation}
and
\begin{equation}
m_{2n}(a,q)\geq a^{n}q^{-n^{2}},\ \ m_{2n+1}(a,q)\geq(a+1)a^{n}q^{-n(n+1)},\ n\in\mathbb{Z}_{+}.\label{eq:lower_estim_moment}
\end{equation}
\end{proposition} 

\begin{proof} It is clear, for instance from (\ref{eq:quadr_recur_moments}),
that each moment $m_{n}(a,q)$ is a polynomial in $a$ and $q^{-1}$
with nonnegative integer coefficients. Furthermore, by the very definition
(\ref{eq:omega_Hn}), $\omega_{n}(a,q)$ is a polynomial in $a$ of
degree $n$ with positive coefficients. From (\ref{eq:lin_recur_moments})
it is seen that
\[
m_{n}(a;q)\leq\frac{\omega_{n}(a,q)}{(q;q)_{n-1}}\,,
\]
and then (\ref{eq:estim_omega}) implies (\ref{eq:upper_estim_moment}).

From (\ref{eq:quadr_recur_moments}) one infers that
\[
m_{2n+1}(a,q)\geq aq^{-2n}m_{2n-1}(a,q),\ m_{2n}(a,q)\geq aq^{-2n+1}m_{2n-2}(a,q),\ \text{for}\ n\geq1.
\]
Using these inequalities and proceeding by mathematical induction
one can verify (\ref{eq:lower_estim_moment}). \end{proof}

%%%%%%%%%%%%%%%%%%%%%%%%%%%%%%%%%%%%%%%%%%%%%%%%%%%%%%%%%%%%%%%%%%%%%%%%
\setcounter{section}{1}
\renewcommand{\thesection}{\Alph{section}}
\setcounter{equation}{0} \renewcommand{\theequation}{\Alph{section}.\arabic{equation}}
\setcounter{theorem}{0} \renewcommand{\thetheorem}{\Alph{section}.\arabic{theorem}}

\section*{Appendix. An asymptotic expansion for the basic confluent
  hypergeometric function and its roots}

The purpose of the appendix is to summarize briefly several useful facts about
the asymptotic behavior of the basic confluent hypergeometric function which
are being referred to in the text of the paper, and to complete them with a
few additional observations while making some details more precise. A common
interpretation of the studied basic hypergeometric function is within the
theory of $q$-Bessel functions but here we prefer, as mentioned already in the
beginning of the paper, to work directly with the function
$\,_{1}\phi_{1}(0;w;q,z)$. Throughout the appendix we assume that $0\leq w<1$
and our focus is on the asymptotic domain $z\to+\infty$.

In \cite{Daalhuis}, Daalhuis derived a remarkable complete asymptotic
expansion of the $q$-Pochhammer symbol. Let us denote
\begin{equation}
  \tilde{q}=e^{4\pi^{2}/\ln(q)},\ \beta(z)=\frac{\pi\ln(z)}{\ln(q)},
  \label{eq:qtilde_beta}
\end{equation}
and
\begin{equation}
  A(z)=2q^{-1/12}\sqrt{z}\,\exp\!\left(-\frac{\ln^{2}(z)}{2\ln(q)}
    +\frac{\pi^{2}}{3\ln(q)}\right)
  \left|\left(\tilde{q}\, e^{-2i\beta(z)};\tilde{q}\right)_{\infty}\right|^{2}.
  \label{eq:Az}
\end{equation}
Then
\begin{equation}
  (z;q)_{\infty}=\frac{A(z)}{(q/z;q)_{\infty}}\,\sin(\beta(z))
  \label{eq:daalhuis}
\end{equation}
for $z>0$. To facilitate comparison of (\ref{eq:daalhuis}) with the original
formula in \cite{Daalhuis} let us note that, for $h>0$ and $\beta$ real,
\begin{eqnarray*}
  \exp\!\left(-\sum_{k=1}^{\infty}
    \frac{\exp(-hk)}{k\sinh(hk)}\,\cos(\beta k)\right)
  & = & \left|\exp\!\left(-\sum_{j=1}^{\infty}\,\sum_{k=1}^{\infty}
      \frac{1}{k}\, e^{(-2hj+i\beta)k}\right)\right|^{2}\\
  \noalign{\smallskip}
  & = & \left|\left(e^{-2h+i\beta};e^{-2h}\right)_{\infty}\right|^{2}.
\end{eqnarray*}

It has also been emphasized in \cite{Daalhuis} that formula
(\ref{eq:daalhuis}) has some useful implications for the theta function. Let
us point out that this relationship between the $q$-Pochhammer symbol and the
theta function works as well in the opposite direction. As far as notations
and basic results related to the theta functions are concerned we refer to
\cite[Chp. XXI]{WhittakerWatson}. The theta function $\vartheta_{1}$ is known
to have the expansion, for $\beta$ real,
\begin{equation}
  \vartheta_{1}(\beta,q)=2q^{1/4}(q^{2};q^{2})_{\infty}
  \sin(\beta)\,\left|\left(e^{2i\beta}q^{2};q^{2}\right)_{\infty}\right|^{2}
  =\sum_{k=0}^{\infty}(-1)^{k}q^{(2k+1)^{2}/4}\sin(\beta(2k+1)).
  \label{eq:theta1}
\end{equation}
The so called Jacobi imaginary transformation of $\vartheta_{1}$ can be
written in the form
\begin{equation}
  \vartheta_{1}(\beta,e^{-\pi h})
  =\frac{1}{i\sqrt{h}}\,\exp\!\left(-\frac{\beta^{2}}{\pi h}\right)
  \vartheta_{1}\!\left(\frac{i\beta}{h},e^{-\pi/h}\right)\!,
  \ \Re h>0.
  \label{eq:Jacobi_imtran}
\end{equation}
Let us note that this identity can be derived, for instance, by applying
Poisson's summation rule to the series expansion in (\ref{eq:theta1}).
Expressing the theta functions occurring in (\ref{eq:Jacobi_imtran}) as
infinite products one obtains
\begin{eqnarray}
  &  & e^{-\pi h/4}\sin(\beta)\,\left(e^{-2\pi h};e^{-2\pi h}\right)_{\infty}
  \left(e^{-2\pi h-2i\beta};e^{-2\pi h}\right)_{\infty}
  \left(e^{-2\pi h+2i\beta};e^{-2\pi h}\right)_{\infty}\nonumber \\
  &  & =\,\frac{1}{\sqrt{h}}\,\exp\!\left(-\frac{\beta^{2}}{\pi h}\right)
  e^{-\pi/(4h)}\sinh\!\left(\frac{\beta}{h}\right)
  \left(e^{-2\pi/h};e^{-2\pi/h}\right)_{\infty}
  \label{eq:jacobi_to_daalhuis}\\
  \noalign{\smallskip} &  & 
  \quad\times\left(e^{-(2\pi+2\beta)/h};e^{-2\pi/h}\right)_{\infty}
  \left(e^{-(2\pi-2\beta)/h};e^{-2\pi/h}\right)_{\infty}.\nonumber 
\end{eqnarray}
Differentiating (\ref{eq:jacobi_to_daalhuis}) with respect to $\beta$ at
$\beta=0$ one derives a rather neat identity for the $q$-Pochhammer symbol,
\begin{equation}
  \left(e^{-2\pi h};e^{-2\pi h}\right){}_{\!\infty}
  =\frac{1}{\sqrt{h}}\,\exp\!\left(\frac{\pi}{12}
    \left(h-\frac{1}{h}\right)\right)
  \left(e^{-2\pi/h};e^{-2\pi/h}\right){}_{\!\infty},\text{ }\Re h>0.
  \label{eq:ident_qPochhammer}
\end{equation}
Furthermore, letting $h=-2\pi/\ln(q)$ and $\beta=\pi\ln(z)/\ln(q)$ in
(\ref{eq:jacobi_to_daalhuis}), and making use of (\ref{eq:ident_qPochhammer})
one arrives at (\ref{eq:daalhuis}).

One may also note that, with increasing $z$, the term
$\left|\left(\tilde{q}\, e^{-2i\beta(z)};\tilde{q}\right)_{\infty}\right|$
occurring in (\ref{eq:Az}) oscillates (logarithmically) between the extreme
values
\[
\left(\tilde{q};\tilde{q}\right)_{\infty}
=q^{1/24}\,\sqrt{-\frac{\ln(q)}{2\pi}}\,
\exp\!\left(-\frac{\pi^{2}}{6\ln(q)}\right)(q;q)_{\infty}
\]
and
\[
\left(-\tilde{q};\tilde{q}\right)_{\infty}
=\frac{q^{-1/48}}{\sqrt{2}}\,\exp\!\left(-\frac{\pi^{2}}{6\ln(q)}\right)
\left(q^{1/2};q\right)_{\infty}.
\]

By differentiating (\ref{eq:daalhuis}) one obtains the asymptotic formula
\begin{eqnarray}
  \frac{\partial(z;q)_{\infty}}{\partial z}
  & = & \frac{A(z)}{(q/z;q)_{\infty}\, z}
  \Bigg(\!\!\left(-\frac{\beta(z)}{\pi}+\frac{1}{2}
    +O\!\left(\frac{1}{z}\right)\!\right)\sin(\beta(z))
  +\frac{\pi}{\ln(q)}\,\cos(\beta(z))\nonumber \\
  \noalign{\smallskip} &  & \qquad\qquad\quad+\,\frac{8\pi}{\ln(q)}
  \sum_{k=1}^{\infty}\frac{\tilde{q}^{k}}{|1-\tilde{q}^{k}e^{-2i\beta(z)}|^{2}}\,
  \sin^{2}(\beta(z))\cos(\beta(z))\Bigg)
  \label{eq:der_qPochhammer}
\end{eqnarray}
as $z\to+\infty$. Note that the $O(z^{-1})$ term is in fact
\[
z\,(q/z;q)_{\infty}\,\frac{\partial}{\partial z}\frac{1}{(q/z;q)_{\infty}}
=-z\,\frac{\partial}{\partial z}\ln((q/z;q)_{\infty})
=-\frac{1}{z}\sum_{j=1}^{\infty}\frac{q^{j}}{1-q^{j}z^{-1}}.
\]

To deduce from (\ref{eq:daalhuis}) and (\ref{eq:der_qPochhammer}) some
information about the asymptotic behavior of the function
$\,_{1}\phi_{1}(0;w;q,z)$ for $z$ large one needs the following fundamental
relation which has been derived in \cite[Prop. 2.1]{KoornwinderSwarttouw},
\begin{equation}
  \,_{1}\phi_{1}(0;w;q,z)
  =\frac{(z;q)_{\infty}}{(w;q)_{\infty}}\,\,_{1}\phi_{1}(0;z;q,w).
  \label{eq:phi_phi}
\end{equation}

\begin{theorem}
\label{thm:asympt_q-confhyp}
Let ($[x]$ standing for the integer part of $x\in\mathbb{R}$)
\begin{equation}
  K(z)=\left[\frac{1}{2}-\frac{\ln(z)}{\ln(q)}\right]\!.
  \label{eq:Kz}
\end{equation}
With the notation introduced in (\ref{eq:qtilde_beta}), (\ref{eq:Az}), and
assuming $0\leq w<1$, there exist functions $B(w,z)$ and $C(w,z)$ such that
\begin{eqnarray*}
 &  & \hskip-1.5em\,_{1}\phi_{1}(0;w;q,z)\,=\,\frac{B(w,z)}{(w;q)_{\infty}}\\
 &  & \ \times\!\left(\! A(z)\sin(\beta(z))+(-1)^{K(z)+1}q^{(K(z)+1)K(z)/2}w^{K(z)+1}\frac{\left(q^{K(z)+1}z;q\right){}_{\infty}}{(q;q)_{\infty}}\, C(w,z)\!\right)
\end{eqnarray*}
and (for a fixed $w$)
\[
B(w,z)=1+O(z^{-1}),\ C(w,z)=1+O(z^{-1})\text{ }\text{as}\ z\to+\infty.
\]
\end{theorem}

\begin{proof}
  In view of (\ref{eq:phi_phi}), we have
\begin{eqnarray}
  &  & \,_{1}\phi_{1}(0;w;q,z)\label{eq:phi_split}\\
  &  & =\,\frac{1}{(w;q)_{\infty}}\!\left((z;q)_{\infty}
    \sum_{k=0}^{K(z)}\frac{(-1)^{k}q^{k(k-1)/2}w^{k}}{(q;q)_{k}(z;q)_{k}}
    +\sum_{k=K(z)+1}^{\infty}\frac{(-1)^{k}q^{k(k-1)/2}
      w^{k}}{(q;q)_{k}}\,(q^{k}z;q)_{\infty}\right)\!.\nonumber 
\end{eqnarray}
Making use of (\ref{eq:daalhuis}) one finds that it suffices to put
\[
B(w,z)=\frac{1}{(q/z;q)_{\infty}}\left(1+\sum_{k=1}^{K(z)}
\frac{(-1)^{k}q^{k(k-1)/2}w^{k}}{(q;q)_{k}(z;q)_{k}}\right)\!,\ C(w,z)
=\frac{\tilde{C}(w,z)}{B(w,z)},
\]
where
\[
\tilde{C}(w,z)=\left(q^{K(z)+2};q\right)_{\infty}
\left(1+\sum_{j=1}^{\infty}\frac{(-1)^{j}
q^{(j-1)j/2}}{\left(q^{K(z)+2};q\right){}_{j}}
\frac{\left(q^{K(z)+1}w\right)^{j}}{\left(q^{K(z)+1}z;q\right){}_{j}}\right)\!.
\]
Note that $q^{1/2}\leq q^{K(z)}z<q^{-1/2}$ and so $q^{K(z)}=O(z^{-1})$.
Furthermore, for $0\leq k\leq K(z)$ one has
\[
\left|(z;q)_{k}\right|\geq q^{k(k-1)/2}(q^{1/2};q)_{\infty}\, z^{k}.
\]
In fact, this is obviously true for $k=0$. For $q^{-1/2}\leq z$ and $1\leq
k\leq K(z)$,
\begin{equation}
  \left|(z;q)_{k}\right|
  =q^{k(k-1)/2}z^{k}(1-z^{-1})(1-q^{-1}z^{-1})\ldots(1-q^{-k+1}z^{-1})
  \geq q^{k(k-1)/2}z^{k}(q^{1/2};q)_{k}.
  \label{eq:qPoch_qz_estim}
\end{equation}
Hence $B(w,z)=1+O(z^{-1})$ as $z\to+\infty$. The rest of the proof is quite
clear.
\end{proof}

\begin{theorem}
\label{thm:asymp_der_q-confhyp}
Under the same assumptions as in Theorem~\ref{thm:asympt_q-confhyp},
\begin{eqnarray*}
  \frac{\partial\,_{1}\phi_{1}(0;w;q,z)}{\partial z}
  & = & \frac{A(z)}{(w;q)_{\infty}\, z}
  \Bigg(\!\left(-\frac{\beta(z)}{\pi}+\frac{1}{2}\right)
  \sin(\beta(z))+\frac{\pi}{\ln(q)}\cos(\beta(z))\\
  \noalign{\smallskip} &  & \qquad\quad+\,\frac{8\pi}{\ln(q)}
  \sum_{k=1}^{\infty}\frac{\tilde{q}^{k}}{|1-\tilde{q}^{k}
    e^{-2i\beta(z)}|^{2}}\,\sin^{2}(\beta(z))\cos(\beta(z))\\
  &  & \qquad\quad+\, O\!\left(\frac{\ln(z)}{z}\right)\!\Bigg)
\end{eqnarray*}
as $z\to+\infty$.
\end{theorem}

\begin{proof}
  Using the same notation as in the proof of
  Theorem~\ref{thm:asympt_q-confhyp} one again starts from equation
  (\ref{eq:phi_split}). Note that
\[
q^{(K(z)+1)K(z)/2}<\exp\!\left(\frac{\ln(q)}{2}
\left(\frac{\ln^{2}(z)}{\ln^{2}(q)}-\frac{1}{4}\right)\right)
=q^{-1/8}\exp\!\left(\frac{\ln^{2}(z)}{2\ln(q)}\right).
\]
Furthermore, for $k>K(z)$ we have $0<(q^{k}z;q)_{\infty}<1$ and
\begin{eqnarray*}
  \left|\frac{\partial}{\partial z}\ln\!\left((q^{k}z;q)_{\infty}\right)\right|
  & = & \sum_{j=k}^{\infty}\frac{q^{j}}{1-q^{j}z}
  \leq\frac{q^{K(z)+1}}{\left(1-q^{1/2}\right)(1-q)}
  <\frac{q^{1/2}}{\left(1-q^{1/2}\right)(1-q)z}.
\end{eqnarray*}
For $1\leq k\leq K(z)$, we again have (\ref{eq:qPoch_qz_estim}) and also
\[
\left|\frac{\partial}{\partial z}\ln((z;q)_{k})\right|
=\frac{1}{z}\sum_{j=0}^{k-1}\frac{q^{j}z}{q^{j}z-1}
\leq\frac{k}{z\left(1-q^{1/2}\right)}.
\]
Consequently,
\[
\frac{\partial\,_{1}\phi_{1}(0;w;q,z)}{\partial z}=\frac{1}{(w;q)_{\infty}}\frac{\partial(z;q)_{\infty}}{\partial z}\!\left(1+O(z^{-1})\right)+\frac{(z;q)_{\infty}}{(w;q)_{\infty}}\, O(z^{-2})+O\!\left(\!\exp\!\left(\frac{\ln^{2}(z)}{2\ln(q)}\right)\!\right)\!.
\]
Recalling (\ref{eq:daalhuis}) and (\ref{eq:der_qPochhammer}) we obtain the
sought formula.
\end{proof}

Zeros of the $q$-Bessel functions have been studied in a number of papers. For
the equation $\,_{1}\phi_{1}(0;w;q,z)=0$ in the complex variable $z$ (with
$0\leq w<1$ being fixed) these results mean that the roots are all positive
and simple \cite{KoelinkSwarttouw}. Ordering the roots increasingly,
$\zeta_{0}<\zeta_{1}<\zeta_{2}<\ldots$, the leading asymptotic term of
$\zeta_{m}$ for large $m$ was derived in
\cite{AbreuBustozCardoso,AnnabyMansour}. In more detail, Annaby and Mansour
showed in \cite[Thm.~2.2]{AnnabyMansour} that
\begin{equation}
  \zeta_{m}=q^{-m}+O(1)\ \text{\ as}\ m\to\infty.
  \label{eq:root_lead_asympt}
\end{equation}
On the basis of Theorem~\ref{thm:asympt_q-confhyp} we can augment this result
by showing that $q^{m}\zeta_{m}$ approaches the value $1$ much faster than one
might guess from (\ref{eq:root_lead_asympt}).

\begin{proposition}
\label{thm:asympt_roots_phi}
Denote by $0<\zeta_{0}<\zeta_{1}<\zeta_{2}<\ldots$ the increasingly ordered
roots of the equation $\,_{1}\phi_{1}(0;w;q,z)=0$ in the variable $z$. Then
\[
\zeta_{m}=q^{-m}-\frac{w^{m+1}q^{m^{2}}}{(q;q)_{\infty}^{\,2}}
\left(1+O(q^{m})\right)\text{ }\text{as}\ m\to\infty.
\]
\end{proposition}

\begin{proof}
  Recalling (\ref{eq:Kz}), one can see from (\ref{eq:root_lead_asympt}) that
  $K(z)=m$ on quite a large neighborhood of $\zeta_{m}$. More precisely, if
\[
q^{-m}(q^{1/2}-1)\leq z-q^{-m}<q^{-m}(q^{-1/2}-1)
\]
then $K(z)=m$. Bearing in mind equation (\ref{eq:root_lead_asympt}) we write
$\zeta_{m}=q^{-m}+\epsilon_{m}$ while assuming $\epsilon_{m}$ to be
bounded. According to Theorem~\ref{thm:asympt_q-confhyp} we have to solve the
equation
\[
\sin\!\left(\frac{\pi\ln\left(1+q^{m}\epsilon_{m}\right)}{\ln(q)}\right)
=\frac{q^{(m+1)m/2}w^{m+1}}{A(q^{-m}+\epsilon_{m})}
\frac{(q+q^{m+1}\epsilon_{m};q)_{\infty}}{(q;q)_{\infty}}\, C(w,q^{-m}+\epsilon_{m}).
\]
Since $A(q^{-m}+\epsilon_{m})=A(q^{-m})\left(1+O(q^{m})\right)$ and
$C(w,q^{-m}+\epsilon_{m})=1+O(q^{m})$, one has
\[
\frac{\pi q^{m}\epsilon_{m}}{\ln(q)}=\frac{q^{(m+1)m/2}
w^{m+1}}{A(q^{-m})}\left(1+O(q^{m})\right).
\]
Recalling (\ref{eq:ident_qPochhammer}) one finds that
\[
A(q^{-m})=-\frac{\ln(q)}{\pi}\, q^{-(m+1)m/2}(q;q)_{\infty}^{\,2}.
\]
The result readily follows.
\end{proof}

\begin{corollary}
  Under the same assumptions as in Proposition~\ref{thm:asympt_roots_phi},
\[
\frac{\partial\,_{1}\phi_{1}(0;w;q,\zeta_{k})}{\partial z}
=(-1)^{k+1}\frac{(q;q)_{\infty}^{\,2}}
{(w;q)_{\infty}}\, q^{-k(k-1)/2}\left(1+O(q^{k})\right)
\text{ }\text{as}\ k\to\infty.
\]
\end{corollary}

\begin{proof}
  The formula follows immediately from Theorem~\ref{thm:asymp_der_q-confhyp}
  and Proposition~\ref{thm:asympt_roots_phi} if taking into account
  (\ref{eq:ident_qPochhammer}).
\end{proof}

\section*{Acknowledgments}

The authors wish to acknowledge gratefully partial support from grant
No. GA13-11058S of the Czech Science Foundation.

\end{document}